\pdfoutput=1
\RequirePackage{ifpdf}
\ifpdf
\documentclass[pdftex]{sigma}
\else
\documentclass{sigma}
\fi

\newcommand\Ga{\Gamma}
\newcommand\iy\infty

\newcommand\CC{\mathbb{C}}
\newcommand\ZZ{\mathbb{Z}}
\newcommand\RR{\mathbb{R}}
\newcommand\const{{\rm const}}
\newcommand{\hyp}[5]{\,\mbox{}_{#1}F_{#2}\!\left(
  \genfrac{}{}{0pt}{}{#3}{#4};#5\right)}
\newcommand\Tstrut{\rule{0pt}{3ex}}         
\newcommand\Bstrut{\rule[-3ex]{0pt}{0pt}}   
\renewcommand\Re{{\rm Re}\,}

\numberwithin{equation}{section}

\newtheorem{Theorem}{Theorem}[section]
\newtheorem{Lemma}[Theorem]{Lemma}
\newtheorem{Proposition}[Theorem]{Proposition}
 { \theoremstyle{definition}
\newtheorem{Remark}[Theorem]{Remark} }

\begin{document}

\newcommand{\arXivNumber}{1504.08144}


\renewcommand{\thefootnote}{$\star$}

\renewcommand{\PaperNumber}{074}

\FirstPageHeading

\ShortArticleName{Transmutation Operators}

\ArticleName{Fractional Integral and Generalized Stieltjes\\
Transforms for Hypergeometric Functions\\ as
Transmutation Operators\footnote{This paper is a~contribution to the Special Issue on Exact Solvability and Symmetry Avatars
in honour of Luc Vinet.
The full collection is available at
\href{http://www.emis.de/journals/SIGMA/ESSA2014.html}{http://www.emis.de/journals/SIGMA/ESSA2014.html}}}

\Author{Tom H.~KOORNWINDER}

\AuthorNameForHeading{T.H.~Koornwinder}

\Address{Korteweg-de Vries Institute, University of Amsterdam,\\
P.O.\ Box 94248, 1090 GE Amsterdam, The Netherlands}
\Email{\href{mailto:T.H.Koornwinder@uva.nl}{T.H.Koornwinder@uva.nl}}
\URLaddress{\url{https://staff.fnwi.uva.nl/t.h.koornwinder/}}

\ArticleDates{Received April 29, 2015, in f\/inal form September 14, 2015; Published online September 20, 2015}

\Abstract{For each of the eight $n$-th derivative parameter changing
formulas for Gauss hypergeometric
functions a corresponding fractional integration formula is given.
For both types of formulas the dif\/ferential or integral operator
is intertwining between two actions
of the hypergeometric dif\/ferential operator
(for two sets of parameters): a so-called transmutation
property. This leads to eight fractional integration formulas and four
generalized Stieltjes transform formulas for
each of the six dif\/ferent explicit solutions of the hypergeometric
dif\/ferential equation, by letting the transforms act on the solutions.
By specialization two Euler type integral representations for each
of the six solutions are obtained.}


\Classification{33C05; 44A15; 44A20; 26A33}

\renewcommand{\thefootnote}{\arabic{footnote}}
\setcounter{footnote}{0}

\section{Introduction}\label{81}

This paper has two sources of inspiration. The f\/irst aim was to give a complete list
of the fractional integration formulas corresponding to the eight parameter changing
$n$-th derivative formulas for Gauss hypergeometric functions given
in \cite[2.8(20)--(27)]{4} and \cite[\S~15.5]{3}.
The f\/irst fractional generalization of one of these dif\/ferentiation formulas was given
by Bateman \cite[p.~184]{1}. Fractional generalizations of some further dif\/ferentiation
formulas were given by Askey \& Fitch \cite[Section~2]{2}. One still missing case
was partially handled by Camporesi \cite[paragraph after~(2.28)]{5}.
In this paper the full list of the eight fractional integral transformation formulas will
be given.

Another observation leading to this paper was that
{\em Euler's integral representation}
for the Gauss hypergeometric function, when written as a fractional integral
\begin{gather}
\hyp21{a,b}cx=\frac{\Ga(c)}{\Ga(b)\Ga(c-b)}x^{1-c}
\int_0^x y^{b-1}(1-y)^{-a}(x-y)^{c-b-1}\,dy\nonumber\\
\hphantom{\hyp21{a,b}cx=}{} \ (0<x<1,\;\Re c>\Re b>0),
\label{64}
\end{gather}
should have a proof by using the hypergeometric dif\/ferential equation
\cite[(2.3.5)]{8}
\begin{gather}
L_{a,b,c;z} \left(\hyp21{a,b}cz\right)=0,
\label{79}
\end{gather}
where
\begin{gather}
L_{a,b,c;z}\big(f(z)\big)=
(L_{a,b,c}f)(z):=z(1-z)f''(z)+(c-(a+b+1)z)f'(z)-abf(z),
\label{110}
\end{gather}
and that then
essentially the same proof should also yield that
\begin{gather}
f(x):=|x|^{1-c}\int_m^M |y|^{b-1}|1-y|^{-a}|x-y|^{c-b-1}\,dy
\label{76}
\end{gather}
is a solution of the hypergeometric dif\/ferential equation if $m$, $M$ and $x$ are as follows:
\\[\smallskipamount]
{\footnotesize
\centerline{\begin{tabular}
{| r || c | c | c | c | c | c | c | c | c | c |}
\hline
$(m,M)$&$(-\iy,x)$&$(-\iy,0)$&$(x,0)$&$(0,x)$&$(0,1)$&$(0,1)$&$(x,1)$&
$(1,x)$&$(1,\iy)$&$(x,\iy)$\\
\hline
$x\in$&$(-\iy,0)$&$(0,\iy)$&$(-\iy,0)$&$(0,1)$&$(-\iy,0)$&$(1,\iy)$&
$(0,1)$&$(1,\iy)$&$(-\iy,1)$&$(1,\iy)$\\
\hline
\end{tabular}}}
\\[\smallskipamount]
This means that $m$ and $M$ in~\eqref{76} are two consecutive points
of singularity of the integrand.
Indeed, we will see that in all listed cases the right-hand side
of~\eqref{76} equals a constant multiple of one of the six explicit
solutions $w_j$ \cite[\S~15.10(ii)]{3} of the hypergeometric
dif\/ferential equation.
In fact, as was Sergei Sitnik kindly commenting to me following
an earlier version of this paper, the above observations (in the
fractional integral case where~$m$ or~$M$ equals~$x$) were already
made in great detail by Letnikov~\cite{17} in~1874 in a paper
which seems to have been unobserved outside Russia.

The Euler type integral representations of fractional integral
type are specializations
of parameter changing fractional integral transforms acting
on some~$w_j$.
For instance, \eqref{64} is a~special case of {\em Bateman's
fractional integral formula} \cite[p.~184]{1}:
\begin{gather}
\int_0^x y^{c-1}
\hyp21{a,b}cy \frac{(x-y)^{\mu-1}}{\Ga(\mu)}\,dy
=\frac{\Ga(c)}{\Ga(c+\mu)} x^{c+\mu-1}\hyp21{a,b}{c+\mu}x\nonumber\\
\hphantom{\int_0^x y^{c-1}
\hyp21{a,b}cy \frac{(x-y)^{\mu-1}}{\Ga(\mu)}\,dy=}{}
(0<x<1,\;\Re c>0,\;\Re\mu>0).
\label{1}
\end{gather}
It will turn out that \eqref{1}, and all other
fractional integral formulas for hypergeometric functions
to be considered, admit a proof by using the hypergeometric
dif\/ferential equation. Here $L_{a,b,c}$, given by~\eqref{110},
occurs in so-called transmutation formulas,
for instance in connection with~\eqref{1}:
\begin{gather}
L_{a,b,c+\mu;x}\left(x^{1-c-\mu}\!\int_0^x y^{c-1} f(y)
\frac{(x-y)^{\mu-1}}{\Ga(\mu)}\,dy\right)
=x^{1-c-\mu}\!\int_0^x y^{c-1} (L_{a,b,c}f)(y)
\frac{(x-y)^{\mu-1}}{\Ga(\mu)}\,dy\nonumber\\
\hspace*{50mm} (f\in C^2([0,1)),\;x\in(0,1),\;\Re c>0,\;\Re\mu>0).
\label{56}
\end{gather}
Transmutation is a term which occurs in many meanings in science,
and even can have various meanings in mathematics, but in the sense
used here it was f\/irst considered in great detail by Lions~\cite{11}, namely
as an operator~$A$, often an integral operator, intertwining
between two dif\/ferential operators~$L_1$ and~$L_2$:
\begin{gather*}
L_1 A=A L_2.
\end{gather*}
In most examples, but not here, $L_1=d^2/dx^2$.
Then a typical example for $L_2$ would be the Bessel type
dif\/ferential operator $L_2=d^2/dx^2 + a x^{-1} d/dx$.
Lions~\cite{11} built on earlier work by J.~Delsarte (1938) and
Levitan~(1951). Many papers and books on transmutation have appeared
since then. See the surveys~\cite{21,22} by Sitnik and
references given there.

The cases of \eqref{76} where $m$ and $M$ are not equal to~$x$
are variants of the {\em generalized Stieltjes transform},
introduced by Widder~\cite[Section~8]{7}, and further developed
by many authors, see for instance references in~\cite{20, 19}.
We consider the generalized Stieltjes transform
as a transform sending~$f$ to~$g$ of the form
\begin{gather}
g(x)=\int_m^M f(y) |y-x|^{\mu-1}\,dy,
\label{75}
\end{gather}
where \looseness=-1 $(m,M)$ is $(-\iy,0)$ or $(0,1)$ or $(1,\iy)$, where
$x$ is on $\RR$ outside the integration interval, and where the
function $y\mapsto |y|^{\mu-1} f(y)$ is $L^1$ on the integration
interval.
The Euler type integral representations of this
form are special cases
of generalized Stieltjes transforms which send a solu\-tion~$w_i$ of the hypergeometric
dif\/ferential equation to a~solution~$w_j$ and change the parameters.

Some of these formulas can be found in literature, notably in
the Bateman project~\cite{15}.
There \cite[14.4(9)]{15} is essentially
the case $(m,M)=(-\iy,0)$ of~\eqref{76}, while
\cite[20.2(10)]{15} is a generalized Stieltjes transform sending a
${}_2F_1$ to a ${}_2F_1$.
A formula by
Karp \& Sitnik \cite[Lemma 2]{6} is essentially a
generalized Stieltjes transform sending a ${}_2F_1$ to a ${}_3F_2$,
which can be specialized to a transform sending ${}_2F_1$ to ${}_2F_1$.
We will give a long list of generalized Stieltjes transforms
sending some~$w_i$ to some~$w_j$, including the two from literature
just mentioned. The formulas in this list are essentially equivalent:
they can all be derived from each other.

The case $\mu=0$ of \eqref{75} is essentially the classical Stieltjes
transform. It will not change the parameters of the hypergeometric
solutions. A well-known example of this case is
\cite[(4.61.4)]{10}, which sends Jacobi polynomials to
Jacobi functions of the second kind.

The idea that formulas for hypergeometric functions can be proved
by using the hypergeometric dif\/ferential equation goes back to Riemann.
See Andrews, Askey \& Roy \cite[Sections~2.3 and~3.9]{8} how this method
can be used for a proof of Pfaf\/f's and Euler's transformation formulas
and of quadratic transformation formulas. Such methods are recently
also used by Paris and coauthors~\cite{13, 12}. They refer to
an earlier proof in this spirit by Rainville \cite[p.~126]{14}
of a~quadratic transformation formula involving a~${}_1F_1$ and
a~${}_0F_1$. In this paper the method will be applied to integral
formulas, but with a dif\/ferent focus.

Our message is that many
formulas for hypergeometric functions have a companion formula
involving the hypergeometric dif\/ferential equation, which is more
universal because it will imply or suggest many formulas involving
the various solutions $w_j$ of the hypergeometric dif\/ferential equation.
A~f\/inal rigorous proof of these formulas may not use the universal
formula (as will be often the case in the present paper), but the
universal formula is helpful for arriving at these formulas and for
organizing them.

Quite probably the ideas of this paper will also work in other situations,
for instance for Appell hypergeometric series.

The contents of this paper are as follows. After some preliminaries
for hypergeometric functions in Section \ref{103},
we illustrate in Section~\ref{104} the main ideas
of the paper for the special case of the Bateman integral~\eqref{1}.
This takes quite a few pages, but it is less technical
than the rest of the paper.
In Section~\ref{11} we state and prove the eight fractional integral
transformations corresponding to the eight $n$-th derivative formulas
for hypergeometric functions. We give also eight corresponding
transmutation formulas. In Section~\ref{98}
we discuss the 48 fractional
integration formulas for the six solutions $w_j$,
which can be obtained by rewriting the formulas in Section~\ref{11}.
Not all formulas will be given explicitly.
In Section~\ref{99} we give 24 generalized Stieltjes transforms
sending some $w_i$ to some $w_j$.
In Section~\ref{105} we give for each of the
six solutions $w_j$ two Euler type integral representations.
They can all be obtained by specialization of formulas in
Sections~\ref{98} and~\ref{99}.
Finally, Section~\ref{106} discusses the connection between
generalized Stieltjes transforms of dif\/ferent order by fractional
integration, and how this leads to connections between formulas in
Sections~\ref{98} and~\ref{99}.

\section{Preliminaries about hypergeometric functions}
\label{103}
The {\em Gauss hypergeometric function} \cite[Chapter~2]{4},
\cite[Chapter~2]{8}, \cite[Chapter~15]{3}
is def\/ined by its power series
\begin{gather}
\hyp21{a,b}cz=F(a,b;c;z)
:=\sum_{k=0}^\iy \frac{(a)_k (b)_k}{(c)_k k!} z^k
\qquad(|z|<1).
\label{38}
\end{gather}
Here the (complex) parameters $a,b,c$ are taken generically.
In particular the series should be well-def\/ined ($c\notin\ZZ_{\le0}$),
while possibly nicer results in case of terminating series
($a$ or $b\in\ZZ_{\le0}$) are not paid attention to.
The function \eqref{38} has a one-valued analytic continuation to $\CC\backslash[1,\iy)$.
It satisf\/ies the {\em hypergeometric differential equation}~\eqref{79}
and it is uniquely determined as the function regular near~0 and equal to~1
at 0 which is annihilated by $L_{a,b,c}$.

Important transformation formulas, due to Pfaf\/f and Euler, respectively,
are \cite[(2.2.6), (2.2.7)]{8}, \cite[(15.8.1)]{3}:
\begin{alignat}{3}
& \hyp21{a,b}cz =(1-z)^{-a}\hyp21{a,c-b}c{\frac z{z-1}}\qquad&&
(z\notin[1,\iy)),&
\label{61}\\
& \hyp21{a,b}cz =(1-z)^{c-a-b}\hyp21{c-a,c-b}cz\qquad && (z\notin[1,\iy)).&
\label{62}
\end{alignat}
Often, a specif\/ic formula for $F(a,b;c;\,.\,)$ trivially implies another
one by the symmetry in $a$ and $b$. For instance, when we will refer to~\eqref{61}, we may also mean a similar identity with the
right-hand side given by
$(1-z)^{-b} F(c-a,b;c;z/(z-1))$.
An elementary special case of the hypergeometric function is
\begin{gather}
\hyp21{a,b}bz=(1-z)^{-a}\qquad(z\notin[1,\iy)).
\label{63}
\end{gather}

In the case of generic parameters there are essentially six dif\/ferent
explicit solutions of the Gauss dif\/ferential equation
$L_{a,b,c}f=0$
\cite[Section~2.9]{4}, \cite[\S~15.10(ii)]{3}. These are one-valued
analytic functions on the complex plane with a suitable real interval as
cut:
\begin{alignat}{3}
& w_1(z;a,b,c):=\hyp21{a,b}cz\qquad && (z\notin[1,\iy)),\hspace*{-10mm}&
\label{48}\\
& w_2(z;a,b,c) :=z^{1-c}\hyp21{a-c+1,b-c+1}{2-c}z\quad &&
(z\notin(-\iy,0]\cup[1,\iy)),\hspace*{-10mm}&
\label{49}\\
& w_3(z;a,b,c) :=z^{-a}\hyp21{a,a-c+1}{a-b+1}{z^{-1}}\qquad &&
(z\notin(-\iy,1]),\hspace*{-10mm}&
\label{50}\\
& w_4(z;a,b,c) :=z^{-b}\hyp21{b,b-c+1}{b-a+1}{z^{-1}}\qquad &&
(z\notin(-\iy,1]),\hspace*{-10mm} &
\label{51}\\
& w_5(z;a,b,c) :=\hyp21{a,b}{a+b-c+1}{1-z}\qquad && (z\notin(-\iy,0]),\hspace*{-10mm}&
\label{52}\\
& w_6(z;a,b,c) :=(1-z)^{c-a-b}\hyp21{c-a,c-b}{c-a-b+1}{1-z}\qquad &&
(z\notin(-\iy,0]\cup[1,\iy)).\hspace*{-10mm}&
\label{53}
\end{alignat}
Here we had to exclude not just the branch cuts of the~${}_2F_1$'s, but
also those of the power factors
(we assume principal values for the complex powers).
In the case of~$w_2$, $w_3$, $w_4$, $w_6$ we might have chosen the
branch cuts due to the power factors dif\/ferently. For instance, a
companion of~$w_2$ would be
\begin{gather*}
\widetilde w_2(z;a,b,c):=(-z)^{1-c}\hyp21{a-c+1,b-c+1}{2-c}z\qquad
(z\notin[0,\iy)).
\end{gather*}

We will also consider the six solutions on subintervals of the real axis
including the branch cuts of the power factors as below (here we abuse notation by not changing
it compared to above)
\begin{alignat}{3}
& w_1(x;a,b,c)=\hyp21{a,b}cx  && (x\in(-\iy,1)),\hspace*{-10mm}&
\label{15}\\
& w_2(x;a,b,c) =|x|^{1-c}\hyp21{a-c+1,b-c+1}{2-c}x  &&
(x\in(-\iy,0)\cup(0,1)),\hspace*{-10mm}&
\label{44}\\
& w_3(x;a,b,c) =|x|^{-a}\hyp21{a,a-c+1}{a-b+1}{x^{-1}} &&
(x\in(-\iy,0)\cup(1,\iy)),\hspace*{-10mm}&
\label{45}\\
& w_4(x;a,b,c) =|x|^{-b}\hyp21{b,b-c+1}{b-a+1}{x^{-1}}  &&
(x\in(-\iy,0)\cup(1,\iy)),\hspace*{-10mm}&
\label{46}\\
& w_5(x;a,b,c) =\hyp21{a,b}{a+b-c+1}{1-x}  && (x\in(0,\iy)),\hspace*{-10mm}&
\label{47}\\
& w_6(x;a,b,c) =|1-x|^{c-a-b}\hyp21{c-a,c-b}{c-a-b+1}{1-x} \qquad  &&
(x\in(0,1)\cup(1,\iy)).\hspace*{-10mm}&
\label{16}
\end{alignat}

\section{The main ideas illustrated in a special case}
\label{104}

\subsection[Transmutation property of a differentiation operator]{Transmutation property of a dif\/ferentiation operator}

In \cite[2.8(20)--(27)]{4} or
\cite[\S~15.5]{3} there is a list of eight parameter changing
$n$-th derivative formulas for Gauss hypergeometric functions.
One of these is \cite[(15.5.4)]{3}:
\begin{gather}
\left(\frac d{dz}\right)^n\left(z^{c-1} \hyp21{a,b}cz\right)
=\frac{\Ga(c)}{\Ga(c-n)}\,z^{c-n-1}\hyp21{a,b}{c-n}z.
\label{39}
\end{gather}
Formula \eqref{39} can be proved immediately by power series expansion~\eqref{38}. Note also that the $n$-th derivative case follows by iteration of
the case $n=1$. The case $n=1$ can be rewritten as
\begin{gather}
D_{c-1} \hyp21{a,b}c{\,.\,}=(c-1) \hyp21{a,b}{c-1}{\,.\,},
\label{40}
\end{gather}
where
\begin{gather}
(D_a f)(z):=z f'(z)+a f(z)=z^{-a+1} \frac d{dz}\big(z^a f(z)\big).
\end{gather}
Clearly, if $f$ is analytic at 0 then $D_af$ is analytic at 0 and
$(D_a f)(0)=a f(0)$.

Straightforward computation followed by iteration gives the following
transmutation pro\-perty:
\begin{gather}
L_{a,b,c-1} D_{c-1} =D_{c-1} L_{a,b,c},
\label{41}\\
L_{a,b,c-n} D_{c-n}\cdots D_{c-2} D_{c-1} =
D_{c-n}\cdots D_{c-2} D_{c-1} L_{a,b,c}.
\label{42}
\end{gather}
Formula \eqref{40} is also a consequence of \eqref{41}.
Indeed, by \eqref{41}  $L_{a,b,c-1}$ annihilates the left-hand side
of \eqref{40}.
Since this left-hand side is regular at~0, it must be a constant times
$F(a,b;c-1;\,.\,)$ with the constant obtained by evaluating both sides
at~0.
Similarly, \eqref{39} is a consequence of~\eqref{42}.

Because of the above argument, it is natural to consider
$D_{c-1} w_j(\,.\,;a,b,c)$ not just for $j=1$ but also for the other
explicit solutions of the hypergeometric dif\/ferential equation
($j=2,\ldots,6$). It will turn out that for all~$j$ we get
some constant factor times $w_j(\,.\,;a,b,c-1)$. For instance,
with $w_2$ given by \eqref{49},
\begin{gather}
D_{c-1} w_2(\,.\,;a,b,c)=\frac{(a-c+1)(b-c+1)}{2-c} w_2(\,.\,;a,b,c-1).
\label{59}
\end{gather}
This is equivalent to the prototypical dif\/ferentiation formula for
the hypergeometric function
\cite[(15.5.4)]{1}:
\begin{gather}
\frac d{dz} \hyp21{a,b}cz=\frac{ab}c \hyp21{a+1,b+1}{c+1}z.
\label{60}
\end{gather}

Corresponding to each of the eight $n$-th derivative formulas
\cite[(15.5.2)--(15.5.9)]{3} one can write down a transmutation formula
like~\eqref{42}, and corresponding to each of these transmutation
formulas one can write down an $n$-th derivative formula for each of
the six solutions~$w_j$. These will not be included in the present paper. The transmutation formulas for the f\/irst derivative operators
were earlier given by Derezi\'nski \cite[end of Section~3.4]{23}.
He calls them commutation relations and he relates them to
to the Lie algebra~$\mathfrak{so}(6)$ (or equivalently~$\mathfrak{sl}(4)$).
This way to associate a Lie algebra (the so-called dynamical
symmetry algebra) with a hypergeometric dif\/ferential equation
was f\/irst described by Miller~\cite{24}.
This was later understood in the framework of $\cal A$-hypergeometric
systems, see Saito~\cite{25}.

\begin{Remark}
The transformation formulas \eqref{61} and \eqref{62} also have a more universal background
because they can be understood from transformation properties of the dif\/ferential ope\-ra\-tor~$L_{a,b,c}$
(see also \cite[(2.3.10B) and (2.3.10F)]{8}):
\begin{gather}
 L_{a,b,c;z}\left((1-z)^{-a} f\left(\frac z{z-1}\right)\right)
=-(1-z)^{-a-1}\big(L_{a,c-b,c}f\big)\left(\frac z{z-1}\right),\\
 L_{a,b,c;z}\big((1-z)^{c-a-b} f(z)\big)
=(1-z)^{c-a-b} (L_{c-a,c-b,c}f)(z).
\end{gather}
This suggests identities involving $w_i(z;a,b,c)$,
$(1-z)^{-a} w_j(\frac z{z-1};a,c-b,c)$,
$(1-z)^{-b} w_k(\frac z{z-1};$ $c-a,b,c)$ and $(1-z)^{c-a-b} w_l(z;c-a,c-b,c)$.
These can indeed be given, but
sometimes one needs another branch cut for the power factor than chosen
in \eqref{48}--\eqref{53}.
\end{Remark}

\subsection{Transmutation property of a fractional integral operator}
\label{74}
Both Riemann--Liouville and Weyl fractional integrals will occur in
our formulas. See, for instance, \cite[pp.~181--183]{15} or~\cite{16}.

Formula \eqref{39} can equivalently be written as a repeated integral,
which can be condensed to a single integral \cite[p.~111]{8}. For this purpose, take
$0<x<1$ and assume $\Re c>n$. We obtain
\begin{gather}
x^{c-1} \hyp21{a,b}cx=\frac{\Ga(c)}{\Ga(c-n)}
\int_0^x y^{c-n-1} \hyp21{a,b}{c-n}y \frac{(x-y)^{n-1}}{(n-1)!}\,dy.
\label{43}
\end{gather}
Note that there are no other terms because, by our assumption,
$y^{c-m}$ vanishes at $y=0$ for $m=1,\ldots,n$.
Formula~\eqref{43} has a fractional extension for~$n$ complex with
$\Re c>\Re n>0$. In rewritten form this  is Bateman's fractional integral formula~\eqref{1}, which is
often written in the form
\begin{gather}
\int_0^1 t^{c-1}
\hyp21{a,b}c{tz} \frac{(1-t)^{\mu-1}}{\Ga(\mu)}\,dt
=\frac{\Ga(c)}{\Ga(c+\mu)} \hyp21{a,b}{c+\mu}z\nonumber\\
\hphantom{\int_0^1 t^{c-1}
\hyp21{a,b}c{tz} \frac{(1-t)^{\mu-1}}{\Ga(\mu)}\,dt
=}{} \
(z\in\CC\backslash[1,\iy),\;\Re c>0,\;\Re\mu>0).
\label{8}
\end{gather}
Then \eqref{1} follows from \eqref{8} by specializing $z$ to $x\in(0,1)$
and substituting~$t=y/x$.
By analytic continuation it is suf\/f\/icient to prove \eqref{8} for $|z|<1$.
There it follows by power series expansion~\eqref{38}.
For~\eqref{8} see  also \cite[(2.4)]{2}, \cite[(3.5)]{9} and
\cite[Theorem~2.2.4, (2.9.6)]{8}.

Just as we proved \eqref{40} by~\eqref{41}, we can prove~\eqref{1}
from the following identity (derived by straightforward computation):
\begin{gather}
y^{1-c} L_{1-a,1-b,2-c;y}\big(y^{c-1}(x-y)^{\mu-1}\big)
=x^{c+\mu-1} L_{a,b,c+\mu;x}\big(x^{1-c-\mu}(x-y)^{\mu-1}\big).
\label{54}
\end{gather}
We also need (straightforward by integration by parts):

\begin{Lemma}
\label{55}
Let $f,g\in C^2((m,M))$. Then, for $x\in(m,M)$,
\begin{gather*}
(L_{a,b,c}f)(x) g(x)-f(x) (L_{1-a,1-b,2-c}g)(x)\\
\qquad {}=\frac d{dx}\big(x(1-x)\big(f'(x)g(x)-f(x)g'(x)\big)+
\big(c-1+(1-a-b)x\big)f(x)g(x)\big).
\end{gather*}
Furthermore, let
$x(1-x)f''(x)g(x)$,
$x(1-x)f(x)g''(x)$,
$(1+|x|)f'(x)g(x)$,
$(1+|x|)f(x)g'(x)$,
$f(x)g(x)$,
as functions of $x$, be in $L^1((m,M))$, and assume that
\begin{gather*}
\left.\begin{matrix}
x(1-x)f'(x)g(x)\\[\smallskipamount]
x(1-x)f(x)g'(x)\\[\smallskipamount]
(1+|x|)f(x)g(x)
\end{matrix}\right\}
\to0\quad\mbox{\rm as $x\downarrow m$ or $x\uparrow M$.}
\end{gather*}
Then
\begin{gather}
\int_m^M (L_{a,b,c}f)(x) g(x)\,dx
=\int_m^M f(x) (L_{1-a,1-b,2-c}g)(x)\,dx.
\label{109}
\end{gather}
\end{Lemma}

In particular, $L_{1-a,1-b,2-c}$ is the formal adjoint of $L_{a,b,c}$.

\begin{proof}[Proof of \eqref{1} by \eqref{54}]
First assume $0<x<1$, $\Re c>1$, $\Re\mu>2$,
We have
\begin{gather*}
 L_{a,b,c+\mu;x}\big(x^{1-c-\mu}\times
\mbox{left-hand side of \eqref{1}}\big)\\
\qquad{}=\frac1{\Ga(\mu)}\int_0^x y^{c-1} \hyp21{a,b}cy
L_{a,b,c+\mu;x}\big(x^{1-c-\mu}(x-y)^{\mu-1}\big)\,dy\\
\qquad{}=\frac{x^{1-c-\mu}}{\Ga(\mu)}\int_0^x \hyp21{a,b}cy
L_{1-a,1-b,2-c;y}\big(y^{c-1}(x-y)^{\mu-1}\big)\,dy\\
\qquad{}=\frac{x^{1-c-\mu}}{\Ga(\mu)}\int_0^x
L_{a,b,c;y}\left(\hyp21{a,b}cy\right) y^{c-1} (x-y)^{\mu-1}\,dy=0.
\end{gather*}
For the f\/irst equality sign we twice used a generalized Leibniz rule
\cite[(1.1)]{26} (allowed because the integral in the second line
converges absolutely) together with the vanishing of the integrand of
the left-hand side of~\eqref{1}
for $y=x$ (and similarly for the derivative with respect to $x$ of that integrand)\footnote{The involved dif\/ferentiation
under the integral sign got some fame as ``Feynman's trick'',
see \url{https://en.wikipedia.org/wiki/Differentiation_under_the_integral_sign\#Popular_culture}.}.
For the second equality apply~\eqref{54} and for the third equality
\eqref{109} (the conditions are satisf\/ied).
Because
\begin{gather*}
x^{1-c-\mu}\times\mbox{left-hand side of \eqref{1}}=
\int_0^1 t^{c-1} \hyp21{a,b}c{tx} \frac{(1-t)^{\mu-1}}{\Ga(\mu)}\,dt,
\end{gather*}
this is analytic in $x$ at $x=0$, and for $x=0$ it takes the value
\begin{gather*}
\int_0^1 t^{c-1} \frac{(1-t)^{\mu-1}}{\Ga(\mu)}\,dt=
\frac{\Ga(c)}{\Ga(c+\mu)}.
\end{gather*}
Thus, by the characterization of the hypergeometric function as solution
of the hypergeometric dif\/ferential equation, we have proved~\eqref{54}
for $\Re\mu>2$, $\Re c>1$. These conditions can be relaxed by analytic
continuation.
\end{proof}

If we replace in the above proof the ${}_2F_1$ by a suitable function~$f$
then we obtain the transmutation property~\eqref{56}.

\begin{Remark}
For $\Re\mu>2$ \eqref{56} is equivalent to~\eqref{54}.
For $\mu=3,4,\ldots$ \eqref{56} is also equivalent to~\eqref{42}.
Furthermore, it is suf\/f\/icient to prove~\eqref{54} for $\mu=3,4,\ldots$
because, after division of both sides by~$(x-y)^{\mu-1}$, the two sides
depend polynomially on~$\mu$.
\end{Remark}

Formally, the proof of \eqref{56} can be extended to showing that
\begin{gather}
L_{a,b,c+\mu;x}\left(|x|^{1-c-\mu}\int_m^M |y|^{c-1} f(y)
\frac{|x-y|^{\mu-1}}{\Ga(\mu)}\,dy\right)\nonumber\\
\qquad{}=|x|^{1-c-\mu}\int_m^M |y|^{c-1} (L_{a,b,c}f)(y)
\frac{|x-y|^{\mu-1}}{\Ga(\mu)}\,dy.
\label{57}
\end{gather}
Here $m$ or $M$ may be equal to $x$, but not necessarily, and
$y$ and $x-y$ should not change sign for $y\in[m,M)$.
This becomes rigorous if $f$, $c$ and~$\mu$ are such that the f\/irst and third equality (suitably
modif\/ied) in the above Proof remain valid in the case of~\eqref{57}.
In particular, $f$ should have suitable vanishing properties at~$m$ and~$M$,
even stronger if~$m$ or~$M$ are not equal to~$0$ or~$x$.

If \eqref{57} holds and if $L_{a,b,c} f=0$ then
\begin{gather}
L_{a,b,c+\mu} g=0,\qquad
g(x):=|x|^{1-c-\mu}\int_m^M |y|^{c-1} f(y)
\frac{|x-y|^{\mu-1}}{\Ga(\mu)}\,dy.
\label{58}
\end{gather}
In this paper many explicit cases of \eqref{58} will be given, where
$f=w_i$ for some $i$ and $g=\const\cdot  w_j$ for some~$j$, with $j=i$ if~$m$ or~$M$ equals~$x$. As an example it can be derived that
\begin{gather}
\int_{-\iy}^x (-y)^{c-1} w_2(y;a,b,c) \frac{(x-y)^{\mu-1}}{\Ga(\mu)}\,dy
=\frac{\Ga(a-c-\mu+1)}{\Ga(a-c+1)} \frac{\Ga(b-c-\mu+1)}{\Ga(b-c+1)}
\frac{\Ga(2-c)}{\Ga(2-c-\mu)}\nonumber\\
\hphantom{\int_{-\iy}^x (-y)^{c-1} w_2(y;a,b,c) \frac{(x-y)^{\mu-1}}{\Ga(\mu)}\,dy=}{}
\times (-x)^{c+\mu-1} w_2(x;a,b,c+\mu)\\
\hspace*{50mm}{}
(x<0,\; \Re(a-c-\mu+1),\Re(b-c-\mu+1),\Re\mu>0).\nonumber
\end{gather}
This is the fractional generalization of the iteration of~\eqref{59}.
It can be equivalently written as
\begin{gather}
\int_{-\iy}^x \hyp21{a,b}cy\frac{(x-y)^{\mu-1}}{\Ga(\mu)}\,dy
=\frac{\Ga(a-\mu)}{\Ga(a)} \frac{\Ga(b-\mu)}{\Ga(b)}
\frac{\Ga(c)}{\Ga(c-\mu)} \hyp21{a-\mu,b-\mu}{c-\mu}x\nonumber\\
\hphantom{\int_{-\iy}^x \hyp21{a,b}cy\frac{(x-y)^{\mu-1}}{\Ga(\mu)}\,dy=}{}
\ (x<0,\;\Re a,\; \Re b>\Re\mu>0).
\label{68}
\end{gather}
This is the fractional generalization of the iteration of~\eqref{60}.
A dif\/ferent proof will be given in the next section.

Another explicit case of~\eqref{58} which we will meet is
\begin{gather}
\int_1^\iy y^{c-1} w_6(y;a,b,c) (y-x)^{\mu-1}\,dy\nonumber\\
\qquad {}
=\frac{\Ga(a-c-\mu+1) \Ga(b-c-\mu+1) \Ga(c-a-b+1)}{\Ga(2-c-\mu) \Ga(1-\mu)}
 x^{c+\mu-1} w_2(x;a,b,c+\mu)\nonumber\\
\qquad \quad (0<x<1,\; \Re(a-c-\mu+1),\; \Re(b-c-\mu+1),\; \Re(c-a-b+1)>0).
\label{69}
\end{gather}
The left-hand side is no longer of fractional integral type,
but it is a generalized
Stieltjes transform, to which we will return in a moment.

\subsection{Euler type integral representations}
When we replace $c$, $\mu$ by $b$, $c-b$ in~\eqref{1} or~\eqref{8} and use~\eqref{63} then we
obtain Euler's integral representation as fractional integral~\eqref{64}, or in its most used form
\begin{gather}
\hyp21{a,b}cz=\frac{\Ga(c)}{\Ga(b)\Ga(c-b)}
\int_0^1 t^{b-1} (1-t)^{c-b-1} (1-tz)^{-a}\,dt\nonumber\\
\hphantom{\hyp21{a,b}cz=}{} \
(z\in\CC\backslash[1,\iy),\;\Re c>\Re b>0).
\label{65}
\end{gather}
The right-hand side of \eqref{64}
is annihilated by $L_{a,b,c;x}$.
This is a consequence of the transmutation property~\eqref{56},
by which $L_{a,b,c;x}$ acting on the right-hand side
of~\eqref{64} is equal to
\begin{gather*}
\const \cdot x^{1-c}\int_0^x y^{b-1} L_{a,b,b;y}\big((1-y)^{-a}\big)
(x-y)^{c-b-1}\,dy=0
\end{gather*}
because $(1-y)^{-a}$ is annihilated by $L_{a,b,b;y}$. This last fact
is also apparent from
\begin{gather*}
L_{a,b,b;y}=\left(y\frac d{dy}+b\right)\circ\left((1-y)\frac d{dy}-a\right).
\end{gather*}

The more general transmutation property \eqref{57}, considered with
$c,\mu$ replaced by~$b$,~$c-b$, suggests that each~$w_j$
has an Euler type integral representation~\eqref{76}.
This is indeed the case, as we already brief\/ly indicated after~\eqref{76}.

We can write \eqref{65} also equivalently as a generalized Stieltjes
transform
\begin{gather}
w_3(x;a,b,c)=\frac{\Ga(a-b+1)}{\Ga(a-c+1)\Ga(c-b)}
\int_0^1 y^{a-c}(1-y)^{c-b-1}(x-y)^{-a}\,dy\nonumber\\
\hphantom{w_3(x;a,b,c)=}{} \
(x>1,\;\Re(a+1)>\Re c>\Re b).
\label{73}
\end{gather}
It will turn out that, more generally and analogous to \eqref{76},
\begin{gather}
g(x)=\int_m^M |y|^{a-c} |1-y|^{c-b-1} |x-y|^{-a}\,dy
\label{85}
\end{gather}
is a solution of the hypergeometric dif\/ferential equation
if~$m$,~$M$ and~$x$ are as listed after~\eqref{76}.

Although most Euler type integral representations of fractional
integral type are equivalent to some integral representation
of generalized Stieltjes transform type by a change of integration
variable, this is no longer true for the integral transforms mapping~$w_i$ to~$w_j$ which specialize to an Euler type
integral representation.

\subsection{Generalized Stieltjes transforms as transmutation operators}

We already observed in Section~\ref{81} that variants of
fractional integral transforms and the Euler integral representation
for hypergeometric functions naturally lead to formulas
involving a~ge\-neralized Stieltjes transform.
In the def\/inition by
Widder~\cite[Section 8]{7} the generalized Stieltjes transform
sends a suitable measure~$\alpha$ or function~$\phi$
(with $d\alpha(t)=\phi(t)\,dt$)
to a function $f$ analytic on $\CC\backslash(-\iy,0]$:
\begin{gather*}
\int_0^\iy \frac{d\alpha(t)}{(z+t)^\rho}=f(z).
\end{gather*}
The special case $\rho=1$ gives the classical Stieltjes transform.
In order to have analytic expressions similar to the ones in fractional
integral transforms, we will work with transforms~\eqref{75}.

Transforms of
generalized Stieltjes type have transmutation properties.
For instance, from \eqref{57} we have, associated with~\eqref{69},
the intertwining property
\begin{gather*}
L_{a,b,c+\mu;x}\left(x^{1-c-\mu}\int_1^\iy y^{c-1} f(y)
(y-x)^{\mu-1}\,dy\right)\\
\qquad {} =x^{1-c-\mu}\int_1^\iy y^{c-1} (L_{a,b,c}f)(y)
(y-x)^{\mu-1}\,dy\qquad(0<x<1).
\end{gather*}
Noteworthy is the case $\mu=0$. Then we have the same hypergeometric dif\/ferential operator on both sides. The Stieltjes transform
will then map solutions of the dif\/ferential equation to other solutions.

Karp \& Sitnik \cite[Lemma~2]{6} proved the following formula:
\begin{gather}
\int_0^1 t^{b-1}(1-t)^{d+e-b-c-1}
\hyp21{d-c,e-c}{d+e-b-c}{1-t}(1-tz)^{-a}\,dt\nonumber\\
\qquad{}
=\frac{\Ga(b)\Ga(c)\Ga(d+e-b-c)}{\Ga(d)\Ga(e)} \hyp32{a,b,c}{d,e}z\nonumber\\
\qquad \quad \ (z\in\CC\backslash[1,\iy),\;\Re b, \Re c,\Re(d+e-b-c)>0).
\label{14}
\end{gather}
If we replace $z$ by $z^{-1}$ then we recognize the formula as a generalized
Stieltjes transform sending a ${}_2F_1$ to a ${}_3F_2$ of general parameters. (For more general formulas expressing
${}_pF_{p-1}$ functions as generalized Stieltjes transforms see~\cite{19} and references given there.)
If moreover $d=a$  the transform in~\cite{14}
sends a ${}_2F_1$ to a ${}_2F_1$:
\begin{gather}
\int_0^1 t^{b-1}(1-t)^{a+e-b-c-1}
\hyp21{a-c,e-c}{a+e-b-c}{1-t}(z-t)^{-a}\,dt\nonumber\\
\qquad {}
=\frac{\Ga(b)\Ga(c)\Ga(a+e-b-c)}{\Ga(a)\Ga(e)} z^{-a} \hyp21{b,c}e{z^{-1}}\nonumber\\
\qquad \quad \ (z\in\CC\backslash(-\iy,1],\;\Re b,\Re c,\Re(d+e-b-c)>0).
\label{86}
\end{gather}
All generalized Stieltjes transforms mapping a solution $w_i$ to a solution~$w_j$ we know can be obtained from~\eqref{86} by change of parameters,
change of integration variable, and application of~\eqref{61}
and~\eqref{62}.

\section[The eight fractional integral transformations of the Gauss
hypergeometric function]{The eight fractional integral transformations\\ of the Gauss
hypergeometric function}\label{11}

Each of the eight $n$-th derivative formulas in \cite[\S~15.5]{3}
has a fractional generalization. Some of these are very well-known, but
others were hardly known until now. They fall apart into three families.
Within a family the formulas follow from each other by application of~\eqref{61} or~\eqref{62}. We will use shorthand names for the eight
cases
of which the meaning will be obvious. The division of the cases into
families and their correspondence with the $n$-th derivative formulas
is as follows:
\\[\smallskipamount]
\centerline{\begin{tabular}{| l || l  l  l | l  l | l  l  l |}
\hline
family&I&I&I&II&II&III&III&III\\
\hline
case&$c+$&$a+$, $c+$& $a+$, $b+$, $c+$ & $a-$& $a+$ & $a-$, $b-$, $c-$&$a-$, $c-$&$c-$\\
\hline
\cite[\S~15.5]{3}&(4)&(8)&(9)&(3)&(5)&(1)&(7)&(6)\\
\hline
\end{tabular}}

\subsection{Family I}
\paragraph{Case $\boldsymbol{c+}$.}
We already discussed Bateman's fractional integral formula~\eqref{1}.
From~\eqref{8} we get another variant of~\eqref{1}:
\begin{gather*}
\int_x^0 (-y)^{c-1}
\hyp21{a,b}cy \frac{(y-x)^{\mu-1}}{\Ga(\mu)}\,dy
=\frac{\Ga(c)}{\Ga(c+\mu)} (-x)^{c+\mu-1}\hyp21{a,b}{c+\mu}x \\
\hphantom{\int_x^0 (-y)^{c-1}
\hyp21{a,b}cy \frac{(y-x)^{\mu-1}}{\Ga(\mu)}\,dy
=} \ (x<0,\;\Re c>0,\;\Re\mu>0),
\end{gather*}
which we can write together with \eqref{1} in a unif\/ied way as:
\begin{gather}
\int_{0<y/x<1} |y|^{c-1}
\hyp21{a,b}cy \frac{|x-y|^{\mu-1}}{\Ga(\mu)}\,dy
=\frac{\Ga(c)}{\Ga(c+\mu)}|x|^{c+\mu-1}\hyp21{a,b}{c+\mu}x\nonumber\\
\qquad (x\in(-\iy,0)\cup(0,1),\;\Re c>0,\;\Re\mu>0).
\label{10}
\end{gather}

\paragraph{Case $\boldsymbol{a+}$, $\boldsymbol{c+}$.}
By using~\eqref{61} in~\eqref{10}
we arrive at
\begin{gather}
\int_{0<y/x<1} |y|^{c-1}(1-y)^{b-c-\mu}
\hyp21{a,b}c y \frac{|x-y|^{\mu-1}}{\Ga(\mu)}\,dy\nonumber\\
\qquad {}
=\frac{\Ga(c)}{\Ga(c+\mu)} |x|^{c+\mu-1}(1-x)^{b-c}\hyp21{a+\mu,b}{c+\mu}x\nonumber\\
\qquad \quad \ (x\in(-\iy,0)\cup(0,1),\;\Re c>0,\;\Re\mu>0).
\end{gather}

\paragraph{Case $\boldsymbol{a+}$, $\boldsymbol{b+}$, $\boldsymbol{c+}$.}
By using~\eqref{62} in~\eqref{10} we arrive at
\begin{gather}
\int_{0<y/x<1} |y|^{c-1}(1-y)^{a+b-c}
\hyp21{a,b}c y \frac{|x-y|^{\mu-1}}{\Ga(\mu)}\,dy\nonumber\\
\qquad {} =\frac{\Ga(c)}{\Ga(c+\mu)} |x|^{c+\mu-1}(1-x)^{a+b-c+\mu} \hyp21{a+\mu,b+\mu}{c+\mu}x\nonumber\\
\qquad \quad(x\in(-\iy,0)\cup(0,1),\;\Re c>0,\;\Re\mu>0).
\label{7}
\end{gather}
See also Askey \& Fitch \cite[(2.11)]{2}.

\subsection{Family~II}
\paragraph{Case $a-$.}
Askey \& Fitch \cite[(2.10)]{2} give
\begin{gather}
\int_0^1 t^{a-\mu-1}
\hyp21{a,b}c{zt}\frac{(1-t)^{\mu-1}}{\Ga(\mu)}\,dt
=\frac{\Ga(a-\mu)}{\Ga(a)} \hyp21{a-\mu,b}cz\nonumber\\
\qquad (z\notin(1,\iy),\;\Re a>\Re\mu>0).
\label{70}
\end{gather}
The proof is as for \eqref{8}. For $|z|<1$ the formula follows by power
series expansion~\eqref{38}, and next the general case
follows by analytic
continuation. Formula~\eqref{70} implies the following fractional
integral formula:
\begin{gather}
\int_{0<y/x<1} |y|^{a-\mu-1}
\hyp21{a,b}cy \frac{|x-y|^{\mu-1}}{\Ga(\mu)}\,dy
=\frac{\Ga(a-\mu)}{\Ga(a)} |x|^{a-1}\hyp21{a-\mu,b}cx\nonumber\\
\qquad (x\in(-\iy,0)\cup(0,1),\;\Re a>\Re\mu>0).
\label{2}
\end{gather}

\paragraph{Case $\boldsymbol{a+}$.}
By using \eqref{61} or \eqref{62} in \eqref{2} we arrive at
\begin{gather}
\int_{0<y/x<1} |y|^{c-a-\mu-1}(1-y)^{a+b-c}
\hyp21{a,b}cy \frac{|x-y|^{\mu-1}}{\Ga(\mu)}\,dy\nonumber\\
\qquad {} =\frac{\Ga(c-a-\mu)}{\Ga(c-a)}  |x|^{c-a-1} (1-x)^{a+b-c+\mu}\hyp21{a+\mu,b}cx\nonumber\\
\qquad\quad \ (x\in(-\iy,0)\cup(0,1),\;\Re(c-a)>\Re\mu>0).
\label{5}
\end{gather}

\subsection{Family III}
\paragraph{Case $\boldsymbol{a-}$, $\boldsymbol{b-}$, $\boldsymbol{c-}$.}
The following generalizes a formula of Camporesi
\cite[paragraph after~(2.28)]{5}.

\begin{Proposition}
We have
\begin{gather}
\int_{-\iy}^x \hyp21{a,b}cy\frac{(x-y)^{\mu-1}}{\Ga(\mu)}\,dy
=\frac{\Ga(a-\mu)}{\Ga(a)} \frac{\Ga(b-\mu)}{\Ga(b)}
\frac{\Ga(c)}{\Ga(c-\mu)} \hyp21{a-\mu,b-\mu}{c-\mu}x\nonumber\\
\qquad (x<1,\;\Re a,\Re b>\Re\mu>0).
\label{4}
\end{gather}
\end{Proposition}

\begin{proof}
Formula \eqref{2} can be rewritten as:
\begin{gather}
\int_{1<\frac yx} |y|^{-a}
\hyp21{a,b}c{y^{-1}} \frac{|x-y|^{\mu-1}}{\Ga(\mu)}\,dy
=\frac{\Ga(a-\mu)}{\Ga(a)} |x|^{-a+\mu}\hyp21{a-\mu,b}c{x^{-1}}\nonumber\\
\qquad (x\in(-\iy,0)\cup(1,\iy),\;\Re a>\Re\mu>0).
\label{72}
\end{gather}
Combination with \cite[2.10(2) and (5)]{4}
(or with \cite[(15.8.2)]{3}) yields \eqref{4} for $x<0$.
This, in its turn, implies by transformation of integration variable and
by analytic continuation that
\begin{gather}
\int_1^\iy \hyp21{a,b}c{1+tz} \frac{(t-1)^{\mu-1}}{\Ga(\mu)}\,dt
=\frac{\Ga(a-\mu)}{\Ga(a)} \frac{\Ga(b-\mu)}{\Ga(b)}
\frac{\Ga(c)}{\Ga(c-\mu)}\nonumber\\
\qquad{} \times(-z)^{-\mu} \hyp21{a-\mu,b-\mu}{c-\mu}{1+z}\qquad
(z\notin[0,\iy),\;\Re a>\Re\mu>0).
\end{gather}
From that we see that \eqref{4} holds for $x<1$.
\end{proof}

\paragraph{Case $\boldsymbol{a-}$, $\boldsymbol{c-}$.}
By using~\eqref{61} in~\eqref{4} we arrive at
\begin{gather}
\int_x^1 (1-y)^{a-\mu-1}
\hyp21{a,b}cy\frac{(y-x)^{\mu-1}}{\Ga(\mu)}\,dy
=\frac{\Ga(a-\mu)}{\Ga(a)}
\frac{\Ga(c-b-\mu)}{\Ga(c-b)} \frac{\Ga(c)}{\Ga(c-\mu)}\nonumber\\
\qquad {}\times(1-x)^{a-1}\hyp21{a-\mu,b}{c-\mu}x\qquad
(x<1,\;\Re a,\Re(c-b)>\Re\mu>0).
\end{gather}

\paragraph{Case $\boldsymbol{c-}$.}
By using \eqref{62} in \eqref{4} we arrive at
\begin{gather}
\int_{-\iy}^x (1-y)^{a+b-c}
\hyp21{a,b}cy\frac{(x-y)^{\mu-1}}{\Ga(\mu)}\,dy
=\frac{\Ga(c-a-\mu)}{\Ga(c-a)}
\frac{\Ga(c-b-\mu)}{\Ga(c-b)} \frac{\Ga(c)}{\Ga(c-\mu)}\nonumber\\
\qquad {} \times(1-x)^{a+b-c+\mu}\hyp21{a,b}{c-\mu}x \qquad\!\!
(x<1,\;\Re(c-a),\Re(c-b)>\Re\mu>0).\!\!\!
\label{6}
\end{gather}

\subsection{Transmutation formulas}
Corresponding to case $c+$ above,
we gave already~\eqref{54}, which gives rise to the transmutation formula~\eqref{56}, and by which
a proof of~\eqref{10} can be given.
Analogues of~\eqref{54} and~\eqref{56} can be given for all cases above.
These have the general form
\begin{gather}
L_{a',b',c';x}\left(\frac{w(y)}{v(x)} |x-y|^{\mu-1}\right)
=L_{1-a,1-b,2-c;y}\left(\frac{w(y)w_2(y)}{v(x)v_2(x)}
|x-y|^{\mu-1}\right)
\label{108}
\end{gather}
and
\begin{gather}
L_{a',b',c';x}\left(\int_I f(y) \frac{w(y)}{v(x)}
|x-y|^{\mu-1}\,dy\right)
=\int_I (L_{a,b,c}f)(y) \frac{w(y)w_2(y)}{v(x)v_2(x)}
|x-y|^{\mu-1}\,dy.
\label{107}
\end{gather}
The data to be used in these two formulas for the eight cases are
specif\/ied in the table below. In general, $w(y)$ is a product of powers
of~$|y|$ and $1-y$, and~$w_2(y)$ is equal to~$|y|$, $1-y$ or $1$.
Similarly for~$v(x)$ and~$v_2(x)$, respectively.
The variable $y$ in \eqref{108} ranges over an open interval~$I$
which is also the integration interval
in~\eqref{107}. The interval~$I$
has endpoint $x$ at one side and
endpoint $x_0=0,1$ or $-\iy$ at the other side. The variable~$x$
in~\eqref{108} and~\eqref{107} ranges over some open subset~$J$ of~$\RR$.
The function~$f$ in~\eqref{107} should be in~$C^2(J)$ and should
moreover satisfy certain growth conditions at~$x_0$, to be specif\/ied
in a moment. The parameter $\mu$ can be arbitrarily complex in~\eqref{108}
but should satisfy $\Re\mu>2$ in~\eqref{107}.
\\[\smallskipamount]
\centerline{\small\begin{tabular}{|@{\,\,}l|l|l|l|l|l|l|l@{\,\,}|}
\hline
case&$a'$&$b'$&$c'$&$x_0$&$w(y)/v(x)$&$w_2(y)/v_2(x)$&$J$
\\
\hline
$c+$&$a$&$b$&$c+\mu$&$0$&$|y|^{c-1}/|x|^{c+\mu-1}$&$1$&$(-\iy,0)\cup(0,1)$
\Tstrut\\[\medskipamount]
$a+$, $c+$&$a+\mu$&$b$&$c+\mu$&$0$&
$\displaystyle\frac{|y|^{c-1}(1-y)^{b-c-\mu}}{|x|^{c+\mu-1}(1-x)^{b-c}}$&
$\displaystyle\frac{1-y}{1-x}$&$(-\iy,0)\cup(0,1)$
\\[\bigskipamount]
$a+$, $b+$, $c+$&$a+\mu$&$b+\mu$&$c+\mu$&$0$&
$\displaystyle\frac{|y|^{c-1}(1-y)^{a+b-c}}{|x|^{c+\mu-1}(1-x)^{b-c}}$&
$1$&$(-\iy,0)\cup(0,1)$\Bstrut\\
\hline
$a-$&$a-\mu$&$b$&$c$&$0$&$|y|^{a-\mu-1}/|x|^{a-1}$&$|y|/|x|$&$(-\iy,0)\cup(0,1)$
\Tstrut\\[\medskipamount]
$a+$&$a+\mu$&$b$&$c$&$0$&
$\displaystyle\frac{|y|^{c-a-\mu-1}(1-y)^{a+b-c}}{|x|^{c-a-1}(1-x)^{a+b-c+\mu}}$&
$|y|/|x|$&$(-\iy,0)\cup(0,1)$
\Bstrut\\
\hline
$a-$, $b-$, $c-$&$a-\mu$&$b-\mu$&$c-\mu$&$-\iy$&$1$&$1$&$(-\iy,1)$
\Tstrut\\[\medskipamount]
$a-$, $c-$&$a-\mu$&$b$&$c-\mu$&$1$&$\displaystyle\frac{(1-y)^{a-\mu-1}}{(1-x)^{a-1}}$&
$\displaystyle\frac{1-y}{1-x}$&$(-\iy,1)$
\\[\bigskipamount]
$c-$&$a$&$b$&$c-\mu$&$-\iy$&$\displaystyle\frac{(1-y)^{a+b-c}}{(1-x)^{a+b-c+\mu}}$&
$1$&$(-\iy,1)$
\Bstrut\\
\hline
\end{tabular}}
\\[\smallskipamount]
Verif\/ication of \eqref{108} in the eight cases is by straightforward
computation, possibly using computer algebra.

In order to f\/ind the growth conditions in~\eqref{107} for~$f$ at~$x_0$
we recall that~\eqref{107} is obtained from the string of equalities
\begin{gather*}
 L_{a',b',c';x}\left(\int_I f(y) \frac{w(y)}{v(x)}
|x-y|^{\mu-1}\,dy\right) =\int_I f(y) L_{a',b',c';x}\left(\frac{w(y)}{v(x)}
|x-y|^{\mu-1}\right)dy\\
\qquad {} =\int_I f(y) L_{1-a,1-b,2-c;y}\left(\frac{w(y)w_2(y)}{v(x)v_2(x)}
|x-y|^{\mu-1}\right)dy\\
\qquad {} =\int_I (L_{a,b,c}f)(y)\frac{w(y)w_2(y)}{v(x)v_2(x)}
|x-y|^{\mu-1}\,dy.
\end{gather*}
In the various steps we have to consider the singularities
for~$y$ at~$x$ and at~$x_0$. The f\/irst singularity is already
taken into account by the condition $\Re\mu>2$.
As for the singularity at~$x_0$ the f\/irst and the third equality
need extra assumptions (which will also imply that the four parts
of the string are well def\/ined). For the f\/irst equality we need
\begin{gather*}
f(y) w(y) (1+|y|)^{\Re \mu-1} \ \mbox{as a function of $y$ is
$L^1$ at $y=x_0$.}
\end{gather*}
For the third equality the needed assumptions follow from
Lemma \ref{55}:
\begin{gather*}
\left.\begin{array}{@{}l}
f''(y) y(1-y)w(y)w_2(y)(1+|y|)^{\Re\mu-1}\\[\smallskipamount]
f'(y) w(y)w_2(y)(1+|y|)^{\Re\mu}\\[\smallskipamount]
f(y) y(1-y)\dfrac{d^2}{dy^2}\big(w(y)w_2(y)(1+|y|)^{\Re\mu-1}\big)\\[\smallskipamount]
f(y) \dfrac d{dy}\big(w(y)w_2(y)(1+|y|)^{\Re\mu}\big)\\[\smallskipamount]
f(y) w(y)w_2(y)(1+|y|)^{\Re\mu-1}
\end{array}\right\}
\;\mbox{as functions of $y$ are
$L^1$ at $y=x_0$,}
\end{gather*}
and
\begin{gather*}
\left.\begin{array}{@{}l}
f'(y) y(1-y)w(y)w_2(y)(1+|y|)^{\Re\mu-1}\\[\smallskipamount]
f(y) y(1-y)\dfrac d{dy}\big(w(y)w_2(y)(1+|y|)^{\Re\mu-1}\big)\\[\smallskipamount]
f(y) w(y)w_2(y)(1+|y|)^{\Re\mu}
\end{array}\right\}
\to0\;\mbox{as $y\to x_0$.}
\end{gather*}

Thus, for given $f$, the identity~\eqref{107} is settled under certain
constraints for $a$, $b$, $c$ and $\mu$. But then the identity will be valid
under relaxed constraints on~$a$, $b$, $c$ and $\mu$ such that both sides of~\eqref{107} are analytic in these four parameters, i.e., if~$\Re\mu>2$
and
\begin{gather*}
\left.\begin{array}{@{}l}
f''(y) y(1-y)w(y)w_2(y)(1+|y|)^{\Re\mu-1}\\[\smallskipamount]
f'(y) w(y)w_2(y)(1+|y|)^{\Re\mu}\\[\smallskipamount]
f(y) w(y) (1+|y|)^{\Re \mu-1}
\end{array}\right\}
\;\mbox{as functions of $y$ are
$L^1$ at $y=x_0$.}
\end{gather*}

\section[Fractional integral transformations for the six solutions
of the hypergeometric differential equation]{Fractional integral transformations for the six solutions\\
of the hypergeometric dif\/ferential equation}
\label{98}

Formula~\eqref{10} is a fractional integral transformation of type $c+$ for the solution
$w_1$ of the hypergeometric dif\/ferential equation.
It turns out that for all six solutions $w_i$ there is such a transformation of
the form $(\,.\,)^{c-1} w_i(\,.\,;a,b,c) \to
(\,.\,)^{c+\mu-1} w_i(\,.\,;a,b,c+\mu)$.
These can all be
obtained  by rewriting  fractional integral transformations for
${}_2F_1(a,b;c;\,.\,)$ of various types given in Section~\ref{11}, namely types
$c+$; $a-$, $b-$, $c-$; $a-$; $a-$; $c-$; $a+$, $b+$, $c+$, respectively.
Here we list these formulas. Each formula is preceded by the formula
number of the formula in Section~\ref{11} to which it reduces:
\begin{gather}
\eqref{10}\colon \ \int_{0<y/x<1} |y|^{c-1}w_1(y;a,b,c) \frac{(x-y)^{\mu-1}}{\Ga(\mu)}\,dy
=\frac{\Ga(c)}{\Ga(c+\mu)}|x|^{c+\mu-1} w_1(x;a,b,c+\mu)\nonumber\\
\hphantom{\eqref{10}\colon \ }{} \qquad\quad (x\in(-\iy,0)\cup(0,1),\;\Re c>0,\;\Re\mu>0);
\label{93}
\\
\eqref{4}\colon \
\int_{-\iy}^x |y|^{c-1} w_2(y;a,b,c) \frac{(x-y)^{\mu-1}}{\Ga(\mu)}\,dy
=\frac{\Ga(a-c-\mu+1)}{\Ga(a-c+1)}\nonumber\\
\hphantom{\eqref{4}\colon \ }{} \qquad{}
\times\frac{\Ga(b-c-\mu+1)}{\Ga(b-c+1)}
\frac{\Ga(2-c)}{\Ga(2-c-\mu)}
|x|^{c+\mu-1} w_2(x;a,b,c+\mu)\nonumber\\
\hphantom{\eqref{4}\colon \ }{} \qquad\quad (x\in(-\iy,1),\;\Re(a-c+1),\Re(b-c+1)>\Re\mu>0);
\label{100}
\\
\eqref{72}\colon \
\int_{y/x>1} |y|^{c-1} w_3(y;a,b,c) \frac{|x-y|^{\mu-1}}{\Ga(\mu)}\,dy
=\frac{\Ga(a-c-\mu+1)}{\Ga(a-c+1)} |x|^{c+\mu-1} w_3(x;a,b,c+\mu)\nonumber\\
\hphantom{\eqref{72}\colon \ }{} \qquad \quad (x\in(-\iy,0)\cup(1,\iy),\;\Re(a-c+1)>\Re\mu>0);
\label{71}
\\
\eqref{72}\colon \
\int_{y/x>1} |y|^{c-1} w_4(y;a,b,c) \frac{|x-y|^{\mu-1}}{\Ga(\mu)}\,dy
=\frac{\Ga(b-c-\mu+1)}{\Ga(b-c+1)} |x|^{c+\mu-1} w_4(x;a,b,c+\mu)\nonumber\\
\hphantom{\eqref{72}\colon \ }{}\qquad \quad (x\in(-\iy,0)\cup(1,\iy),\;\Re(b-c+1)>\Re\mu>0);
\\
\eqref{6}\colon \
\int_x^\iy y^{c-1}w_5(y;a,b,c) \frac{(y-x)^{\mu-1}}{\Ga(\mu)}\,dy
=\frac{\Ga(b-c-\mu+1)}{\Ga(b-c+1)}\nonumber\\
\hphantom{\eqref{6}\colon \ }{} \qquad {}
\times\frac{\Ga(a-c-\mu+1)}{\Ga(a-c+1)} \frac{\Ga(a+b-c+1)}{\Ga(a+b-c-\mu+1)}
x^{c+\mu-1}w_5(x;a,b,c+\mu)\nonumber\\
\hphantom{\eqref{6}\colon \ }{} \qquad \quad
(x\in(0,\iy),\;\Re(b-c+1),\Re(a-c+1)>\Re(\mu)>0);
\label{102}
\\
\eqref{7}\colon \
\int_{0<\frac{1-y}{1-x}<1} y^{c-1} w_6(y;a,b,c) \frac{|y-x|^{\mu-1}}{\Ga(\mu)}\,dy\nonumber\\
\hphantom{\eqref{7}\colon \ }{}\qquad
=\frac{\Ga(c-a-b+1)}{\Ga(c-a-b+\mu+1)} x^{c+\mu-1}w_6(x;a,b,c+\mu)\nonumber\\
\hphantom{\eqref{7}\colon \ }{}\qquad  \quad \
(x\in(0,1)\cup(1,\iy),\;\Re(c-a-b+1)>0,\;\Re\mu>0).
\label{82}
\end{gather}

It is a straightforward exercise to list also the fractional integral
transformation formulas of the~$w_i$ corresponding to the other seven
types. We do not list all these formulas here, but only give their
essential behaviour in the following table. Here an entry in~$i$th row,
$j$th column (place~$(i,j)$)
tells us that the transformation formula for $w_j$ of type given at $(i,1)$
can be reduced to the transformation formula for $w_1$ of type given
at $(i,j)$.
\\[\smallskipamount]
\centerline{\begin{tabular}{|l||l|l|l|l|l|}
\hline
$w_1$&$w_2$&$w_3$&$w_4$&$w_5$&$w_6$\\
\hline
$c+$&$a-$, $b-$, $c-$&$a-$&$a-$&$c-$&$a+$, $b+$, $c+$\\
$a+$, $c+$&$a-$, $c-$&$a+$, $c+$&$a-$, $c-$&$a+$&$a+$\\
$a+$, $b+$, $c+$&$c-$&$a+$&$a+$&$a+$, $b+$, $c+$&$c-$\\
$a-$&$a-$&$a-$, $b-$, $c-$&$c+$&$a-$, $c-$&$a+$, $c+$\\
$a+$&$a+$&$a+$, $b+$, $c+$&$c-$&$a+$, $c+$&$a-$, $c-$\\
$a-$, $b-$, $c-$&$c+$&$a-$&$a-$&$a-$, $b-$, $c-$&$c+$\\
$a-$, $c-$&$a+$, $c+$&$a-$, $c-$&$a+$, $c+$&$a-$&$a-$\\
$c-$&$a+$, $b+$, $c+$&$a+$&$a+$&$c+$&$a-$, $b-$, $c-$\\
\hline
\end{tabular}}
\\[\smallskipamount]
The cases $c+$; $c-$; $a-$, $b-$, $c-$; $a+$, $b+$, $c+$ occur 5 times,
the cases $a+$, $c+$; $a-$, $c-$ 6 times, and the cases~$a-$; $a+$ 8 times.
That these numbers are not all equal will be caused by the symmetry in~$a$ and~$b$ of the hypergeometric function.

Some more examples from the above table which we need later
(because of specialization to Euler type integral representations)
are:
\begin{gather}
\int_{0<\frac yx<1} |y|^{a-\mu-1} w_2(y;a,b,c)
\frac{|x-y|^{\mu-1}}{\Ga(\mu)}\,dy
=\frac{\Ga(a-c-\mu+1)}{\Ga(a-c+1)} |x|^{a-1} w_2(x;a-\mu,b,c)\nonumber\\
\qquad (x\in(-\iy,0)\cup(0,1),\;\Re(a-c+1)>\Re\mu>0),
\label{83}
\\
\int_{\frac yx>1} w_4(y;a,b,c) \frac{|x-y|)^{\mu-1}}{\Ga(\mu)}\,dy
=\frac{\Ga(b-\mu)}{\Ga(b)}  w_4(x;a-\mu,b-\mu,c-\mu)\nonumber\\
\qquad (x\in(-\iy,0)\cup(1,\iy),\;\Re b>\Re\mu>0),
\label{84}
\\
\int_{0<\frac{1-y}{1-x}<1} w_6(y;a,b,c)
\frac{|x-y|^{\mu-1}}{\Ga(\mu)}\,dy
=\frac{\Ga(c-a-b+1)}{\Ga(c-a-b+\mu+1)}
w_6(x;a-\mu,b-\mu,c-\mu)\nonumber\\
\qquad (x\in(0,1)\cup(1,\iy),\;\Re(c-a-b+1)>0,\;\Re\mu>0).
\label{19}
\end{gather}

\section[Generalized Stieltjes transforms between solutions
of the hypergeometric differential equation]{Generalized Stieltjes transforms between solutions\\
of the hypergeometric dif\/ferential equation}
\label{99}

First we recall the proof of~\eqref{14} as given by
Karp \& Sitnik \cite[Lemma~2]{6}. Expand for $|z|<1$
the left-hand side of~\eqref{14} as
\begin{gather*}
\sum_{k=0}^\iy \frac{(a)_k}{k!} z^k
\int_0^1 t^{d+e-b-c-1} (1-t)^{b+k-1}
\hyp21{d-c,e-c}{d+e-b-c}t\,dt.
\end{gather*}
By the limit case $x\uparrow1$ of \eqref{1} this equals
\begin{gather*}
\sum_{k=0}^\iy \frac{(a)_k}{k!} z^k\,
\frac{\Ga(d+e-b-c)\Ga(b+k)}{\Ga(d+e-c+k)}
\hyp21{d-c,e-c}{d+e-c+k}1.
\end{gather*}
By the Gauss summation formula \cite[Theorem 2.2.2]{8} we get
\begin{gather*}
\sum_{k=0}^\iy \frac{(a)_k}{k!} z^k
\frac{\Ga(d+e-b-c)\Ga(b+k)}{\Ga(d+e-c+k)}
\frac{\Ga(c+k)\Ga(d+e-c+k)}{\Ga(d+k)\Ga(e+k)} ,
\end{gather*}
which is equal to the right-hand side of \eqref{14}.
All steps can be rigorously justif\/ied because of the constraints
in \eqref{14}.

We now list the generalized Stieltjes transforms mapping a solution~$w_i$ to a solution~$w_j$. They can all be obtained from the special case~\eqref{86} of~\eqref{14} by change of parameters,
change of integration variable, and application of~\eqref{61}
and~\eqref{62}.
In particular, \eqref{87} and~\eqref{88} below are simple rewritings of
\eqref{86} by a change of parameters. Also, \eqref{90} below is a~rewriting of \cite[20.2(10)]{15}.

The formulas in the list are grouped in cases similar to the cases
in Section~\ref{11}.
In all formulas the constraints are that $x$ is in an
interval $(-\iy,0)$, $(0,1)$ or $(1,\iy)$ which does not coincide
with the integration interval,
and that the arguments of the three gamma functions
in the numerator on the right-hand side have positive real parts.

\paragraph{Case $\boldsymbol{c+}$.}
\begin{gather}
\int_{-\iy}^0 (-y)^{c-1} w_1(y;a,b,c) (x-y)^{\mu-1}\,dy\nonumber\\
\qquad {} =\frac{\Ga(a-c-\mu+1)\Ga(b-c-\mu+1)\Ga(c)}{\Ga(a+b-c-\mu+1)\Ga(1-\mu)}
x^{c-1+\mu} w_5(x;a,b,c+\mu),
\label{90}
\\
\int_1^\iy y^{c-1} w_6(y;a,b,c) (y-x)^{\mu-1}\,dy\nonumber\\
\qquad{}
=\frac{\Ga(a-c-\mu+1) \Ga(b-c-\mu+1) \Ga(c-a-b+1)}{\Ga(2-c-\mu) \Ga(1-\mu)}  |x|^{c+\mu-1} w_2(x;a,b,c+\mu).
\label{101}
\end{gather}

\paragraph{Case $\boldsymbol{a+}$, $\boldsymbol{c+}$.}
\begin{gather}
\int_0^1 y^{c-1} (1-y)^{b-c-\mu} w_1(y;a,b,c) |x-y|^{\mu-1}\,dy\nonumber\\
\qquad {}
=\frac{\Ga(-a-\mu+1)\Ga(b-c-\mu+1)\Ga(c)}{\Ga(b-a-\mu+1)\Ga(1-\mu)}\nonumber\\
\qquad \quad {}\times
  |x|^{c+\mu-1} |x-1|^{b-c} w_4(x;a+\mu,b,c+\mu),
\label{89}
\\
\int_1^\iy y^{c-1}(y-1)^{b-c-\mu} w_3(y;a,b,c) (y-x)^{\mu-1}\,dy\nonumber\\
\qquad {}
=\frac{\Ga(-a-\mu+1) \Ga(b-c-\mu+1) \Ga(a-b+1)}
{\Ga(-c-\mu+2) \Ga(1-\mu)}\nonumber\\
\qquad \quad {}\times
 |x|^{c+\mu-1} |x-1|^{b-c} w_2(x;a+\mu,b,c+\mu).\label{37}
\end{gather}

\paragraph{Case $\boldsymbol{b+}$, $\boldsymbol{c+}$.}
\begin{gather}
\int_0^1 y^{c-1} (1-y)^{a-c-\mu} w_1(y;a,b,c) |x-y|^{\mu-1}\,dy\nonumber\\
\qquad {} =\frac{\Ga(-b-\mu+1)\Ga(a-c-\mu+1)\Ga(c)}{\Ga(a-b-\mu+1)\Ga(1-\mu)}\nonumber\\
\qquad \quad {}\times
 |x|^{c+\mu-1}|x-1|^{a-c} w_3(x;a,b+\mu,c+\mu),
\\
\int_1^\iy y^{c-1}(y-1)^{a-c-\mu} w_4(y;a,b,c) |x-y|^{\mu-1}\,dy\nonumber\\
\qquad {} =\frac{\Ga({-}b-\mu+1) \Ga(a-c-\mu+1) \Ga(b-a+1)}{\Ga(-c-\mu+2) \Ga(1-\mu)}\nonumber\\
\qquad \quad {}\times
|x|^{c+\mu-1}  |x-1|^{a-c} w_2(x;a,b+\mu,c+\mu).
\end{gather}

\paragraph{Case $\boldsymbol{a+}$, $\boldsymbol{b+}$, $\boldsymbol{c+}$.}
\begin{gather}
\int_{-\iy}^0 (-y)^{c-1} (1-y)^{a+b-c}w_1(x;a,b,c) (x-y)^{\mu-1}\,dy\nonumber\\
\qquad{}=\frac{\Ga(-a-\mu+1)\Ga(-b-\mu+1)\Ga(c)}{\Ga(c-a-b-\mu+1)\Ga(1-\mu)}\nonumber\\
\qquad\quad{}\times
 x^{c+\mu-1}(1-x)^{a+b-c+\mu} w_6(x;a+\mu,b+\mu,c+\mu),
\\
\int_1^\iy y^{c-1} (y-1)^{a+b-c} w_5(y;a,b,c)(y-x)^{\mu-1}\,dy\nonumber\\
\qquad{} =\frac{\Ga(-a-\mu+1) \Ga(-b-\mu+1) \Ga(a+b-c+1)}{\Ga(2-c-\mu) \Ga(1-\mu)}\nonumber\\
\qquad\quad {}\times
 |x|^{c+\mu-1} (1-x)^{a+b-c+\mu} w_2(x;a+\mu,b+\mu,c+\mu).
\end{gather}

\paragraph{Case $\boldsymbol{a-}$.}
\begin{gather}
\int_{-\iy}^0 (-y)^{a-\mu-1} w_4(y;a,b,c) (x-y)^{\mu-1}\,dy\nonumber\\
\qquad {} =\frac{\Ga(a-\mu)\Ga(a-c-\mu+1)\Ga(b-a+1)}{\Ga(a+b-c-\mu+1)\Ga(1-\mu)}
  x^{a-1} w_5(x;a-\mu,b,c),
\label{92}
\\
\int_0^1 y^{a-\mu-1} w_6(y;a,b,c) |x-y|^{\mu-1}\,dy\nonumber\\
\qquad {} =\frac{\Ga(a-c-\mu+1) \Ga(a-\mu) \Ga(c-a-b+1)}
{\Ga(a-b-\mu+1) \Ga(1-\mu)}
 |x|^{a-1} w_3(x;a-\mu,b,c).
\label{87}
\end{gather}

\paragraph{Case $\boldsymbol{b-}$.}
\begin{gather}
\int_{-\iy}^0 (-y)^{b-\mu-1} w_3(y;a,b,c)(x-y)^{\mu-1}\,dy\nonumber\\
\qquad {}
=\frac{\Ga(b-\mu)\Ga(b-c-\mu+1)\Ga(a-b+1)}{\Ga(a+b-c-\mu+1)\Ga(1-\mu)}
  x^{b-1} w_5(x;a,b-\mu,c),
\\
\int_0^1 y^{b-\mu-1} w_6(y;a,b,c) |x-y|^{\mu-1}\,dy\nonumber\\
\qquad {} =\frac{\Ga(b-c-\mu+1) \Ga(b-\mu) \Ga(c-a-b+1)}
{\Ga(b-a-\mu+1) \Ga(1-\mu)}
  |x|^{b-1} w_4(x;a,b-\mu,c).
\label{88}
\end{gather}

\paragraph{Case $\boldsymbol{a+}$.}
\begin{gather}
\int_{-\iy}^0 (-y)^{c-a-\mu-1}(1-y)^{a+b-c} w_3(y;a,b,c) (x-y)^{\mu-1}\,dy\\
\qquad{}=\frac{\Ga(-a-\mu+1) \Ga(c-a-\mu) \Ga(a-b+1)}{\Ga(c-a-b-\mu+1)\Ga(1-\mu)}
x^{c-a-1}(1-x)^{a+b-c+\mu} w_6(x;a+\mu,b,c),\nonumber
\\
\int_0^1 y^{c-a-\mu-1} (1-y)^{a+b-c} w_5(y;a,b,c)|x-y|^{\mu-1}\,dy\nonumber\\
\qquad {}=\frac{\Ga(-a-\mu+1) \Ga(c-a-\mu) \Ga(a+b-c+1)}{\Ga(b-a-\mu+1) \Ga(1-\mu)}\nonumber\\
\qquad \quad {}\times
|x|^{c-a-1} |1-x|^{a+b-c-\mu} w_4(x;a + \mu,b,c).
\end{gather}

\paragraph{Case $\boldsymbol{b+}$.}
\begin{gather}
\int_{-\iy}^0 (-y)^{c-b-\mu-1}(1-y)^{a+b-c} w_4(y;a,b,c) (x-y)^{\mu-1}\,dy\nonumber\\
\qquad {}=\frac{\Ga(-b-\mu+1) \Ga(c-b-\mu) \Ga(b-a+1)}
{\Ga(c-a-b-\mu+1)\Ga(1-\mu)}\nonumber\\
\qquad\quad {}\times
x^{c-b-1}(1-x)^{a+b-c+\mu} w_6(x;a,b+\mu,c),
\\
\int_0^1 y^{c-b-\mu-1} (1-y)^{a+b-c} w_5(y;a,b,c) |x-y|^{\mu-1}\,dy\nonumber\\
\qquad{}=\frac{\Ga(-b-\mu+1) \Ga(c-b-\mu) \Ga(a+b-c+1)}{\Ga(a-b-\mu+1) \Ga(1-\mu)}\nonumber\\
\qquad\quad {}\times
|x|^{c-b-1} |1-x|^{a+b-c-\mu} w_3(x;a,b+\mu,c).
\end{gather}

\paragraph{Case $\boldsymbol{a-}$, $\boldsymbol{b-}$, $\boldsymbol{c-}$.}
\begin{gather}
\int_{-\iy}^0 w_2(y;a,b,c) (x-y)^{\mu-1}\,dy\nonumber\\
\qquad{}
=\frac{\Ga(a-\mu)\Ga(b-\mu)\Ga(2-c)}{\Ga(a+b-c-\mu+1)\Ga(1-\mu)}
w_5(x;a-\mu,b-\mu,c-\mu),
\\
\int_1^\iy w_6(y;a,b,c) (y-x)^{\mu-1}\,dy\nonumber\\
\qquad{}
=\frac{\Ga(a-\mu)\Ga(b-\mu)\Ga(c-a-b+1)}{\Ga(c-\mu)\Ga(1-\mu)}
w_1(x;a-\mu,b-\mu,c-\mu).
\end{gather}

\paragraph{Case $\boldsymbol{a-}$, $\boldsymbol{c-}$.}
\begin{gather}
\int_0^1 (1-y)^{a-\mu-1} w_2(y;a,b,c) |x-y|^{\mu-1}\,dy\nonumber\\
\qquad{}
=\frac{\Ga(c-b-\mu)\Ga(a-\mu)\Ga(2-c)}{\Ga(a-b-\mu+1)\Ga(1-\mu)}  |x-1|^{a-1} w_3(x;a-\mu,b,c-\mu),
\\
\int_1^\iy (y-1)^{a-\mu-1} w_4(y;a,b,c) (y-x)^{\mu-1}\,dy\nonumber\\
\qquad{}
=\frac{\Ga(c-b-\mu) \Ga(a-\mu) \Ga(b-a+1)}{\Ga(c-\mu) \Ga(1-\mu)} |1-x|^{a-1} w_1(x;a-\mu,b,c-\mu).
\label{91}
\end{gather}

\paragraph{Case $\boldsymbol{b-}$, $\boldsymbol{c-}$.}
\begin{gather}
\int_0^1 (1-y)^{b-\mu-1} w_2(y;a,b,c) |x-y|^{\mu-1}\,dy\nonumber\\
\qquad{}
=\frac{\Ga(c-a-\mu)\Ga(b-\mu)\Ga(2-c)}{\Ga(b-a-\mu+1)\Ga(1-\mu)}  |x-1|^{b-1} w_4(x;a,b-\mu,c-\mu),
\\
\int_1^\iy (y-1)^{b-\mu-1} w_3(y;a,b,c) (y-x)^{\mu-1}\,dy\nonumber\\
\qquad{}
=\frac{\Ga(c-a-\mu) \Ga(b-\mu) \Ga(a-b+1)}{\Ga(c-\mu) \Ga(1-\mu)} |1-x|^{b-1} w_1(x;a,b-\mu,c-\mu).
\end{gather}

\paragraph{Case $\boldsymbol{c-}$.}
\begin{gather}
\int_{-\iy}^0 (1-y)^{a+b-c} w_2(y;a,b,c) (x-y)^{\mu-1}\,dy\nonumber\\
\qquad{}
=\frac{\Ga(c-a-\mu) \Ga(c-b-\mu) \Ga(2-c)}
{\Ga(c-a-b-\mu+1) \Ga(1-\mu)}  |1-x|^{a+b-c+\mu} w_6(x;a,b,c-\mu),
\\
\int_1^\iy (y-1)^{a+b-c} w_5(y;a,b,c) (y-x)^{\mu-1}\,dy\nonumber\\
\qquad{}
=\frac{\Ga(c-a-\mu)\Ga(c-b-\mu)\Ga(a+b-c+1)}{\Ga(c-\mu)\Ga(1-\mu)} (1-x)^{a+b-c+\mu} w_1(x;a,b,c-\mu).
\end{gather}

We summarize the results of the above list
in the following table. In the box in row~$w_i$ and
column $w_j$ the type is given of the generalized Stieltjes transform sending~$w_i$ to~$w_j$.
\\[\smallskipamount]
\centerline{\begin{tabular}{| l || l | l | l | l | l | l | }
\hline
&$w_1$&$w_2$&$w_3$&$w_4$&$w_5$&$w_6$\\
\hline\hline
$w_1$&&&$b+$, $c+$&$a+$, $c+$&$c+$&$a+$, $b+$, $c+$\\
\hline
$w_2$&&&$a-$, $c-$&$b-$, $c-$&$a-$, $b-$, $c-$&$c-$\\
\hline
$w_3$&$b-$, $c-$&$a+$, $c+$&&&$b-$&$a+$\\
\hline
$w_4$&$a-$, $c-$&$b+$, $c+$&&&$a-$&$b+$\\
\hline
$w_5$&$c-$&$a+$, $b+$, $c+$&$b+$&$a+$&&\\
\hline
$w_6$&$a-$, $b-$, $c-$&$c+$&$a-$&$b-$&&\\
\hline
\end{tabular}}
\\[\smallskipamount]
Due to the denominator factor $\Ga(1-\mu)$ all generalized Stieltjes
transforms above become zero if~$\mu$ is a positive integer satisfying
the constraints.

\section{Euler type integral representations}
\label{105}

Below we give the explicit Euler type integral representations
which have form~\eqref{76} or~\eqref{85}.

\paragraph{Euler type integrals with integrand $\boldsymbol{|x|^{1-c}|y|^{b-1}|1-y|^{-a}|x-y|^{c-b-1}}$.}
\begin{gather}
w_1(x;a,b,c)=\frac{\Ga(c)}{\Ga(b)\Ga(c-b)} |x|^{1-c}
\int_{0<y/x<1} |y|^{b-1}(1-y)^{-a} |x-y|^{c-b-1}\,dy\nonumber\\
\hphantom{w_1(x;a,b,c)=}{} \
(x\in(-\iy,0)\cup(0,1),\;\Re c>\Re b>0),
\label{23}
\\
w_2(x;a,b,c)=\frac{\Ga(2-c)}{\Ga(a-c+1)\Ga(1-a)} |x|^{1-c}
\int_1^\iy y^{b-1}(y-1)^{-a} (y-x)^{c-b-1}\,dy\nonumber\\
\hphantom{w_2(x;a,b,c)=}{} \
(x\in(-\iy,0)\cup(0,1),\;\Re(c-1)<\Re a<1),
\label{24}
\\
w_3(x;a,b,c)=\frac{\Ga(a-b+1)}{\Ga(a-c+1)\Ga(c-b)} |x|^{1-c}
\int_{y/x>1} |y|^{b-1}|y-1|^{-a} |y-x|^{c-b-1}\,dy\nonumber\\
\hphantom{w_3(x;a,b,c)=}{} \
(x\in(-\iy,0)\cup(1,\iy),\;\Re(a-c+1)>0,\;\Re(c-b)>0),
\label{25}
\\
w_4(x;a,b,c)=\frac{\Ga(b-a+1)}{\Ga(-a)\Ga(b+1)} |x|^{1-c}
\int_0^1 y^{b-1}(1-y)^{-a} |x-y|^{c-b-1}\,dy\nonumber\\
\hphantom{w_4(x;a,b,c)=}{}  \
(x\in(-\iy,0)\cup(1,\iy),\;\Re a<0,\;\Re b>-1),
\label{26}
\\
w_5(x;a,b,c)=\frac{\Ga(a+b-c+1)}{\Ga(a-c+1)\Ga(b)} x^{1-c}
\int_{-\iy}^0 (-y)^{b-1}(1-y)^{-a} (x-y)^{c-b-1}\,dy\nonumber\\
\hphantom{w_5(x;a,b,c)=}{} \
(x\in(0,\iy),\;\Re b>0,\;\Re(a-c+1)>0),
\label{27}
\\
w_6(x;a,b,c)=\frac{\Ga(c-a-b+1)}{\Ga(1-a)\Ga(c-b)} x^{1-c}
\int_{0<\frac{1-y}{1-x}<1} y^{b-1}(1-y)^{-a} |x-y|^{c-b-1}\,dy\nonumber\\
\hphantom{w_6(x;a,b,c)=}{}  \
(x\in(0,1)\cup(1,\iy),\;\Re(c-b)>0,\;\Re a<1).
\label{28}
\end{gather}

\paragraph{Euler type integral transformations with integrand
$\boldsymbol{|y|^{a-c} |1-y|^{c-b-1} |x-y|^{-a}}$.}
\begin{gather}
w_1(x;a,b,c)=\frac{\Ga(c)}{\Ga(b)\Ga(c-b)}
\int_1^\iy y^{a-c}(y-1)^{c-b-1} (y-x)^{-a}\,dy\nonumber\\
\hphantom{w_1(x;a,b,c)=}{} \
(x<1,\;\Re c>\Re b>0),
\label{22}
\\
w_2(x;a,b,c)=\frac{\Ga(2-c)}{\Ga(a-c+1)\Ga(1-a)}
\int_{0<y/x<1} |y|^{a-c}(1-y)^{c-b-1} |x-y|^{-a}\,dy\nonumber\\
\hphantom{w_2(x;a,b,c)=}{} \
(x\in(-\iy,0)\cup(0,1),\;2>\Re(a+1)>\Re c),
\label{29}
\\
w_3(x;a,b,c)=\frac{\Ga(a-b+1)}{\Ga(a-c+1)\Ga(c-b)}
\int_0^1 y^{a-c}(1-y)^{c-b-1} |x-y|^{-a}\,dy\nonumber\\
\hphantom{w_3(x;a,b,c)=}{} \
(x\in(-\iy,0)\cup(1,\iy),\;\Re(a+1)>\Re c>\Re b),
\label{30}
\\
w_4(x;a,b,c)=\frac{\Ga(b-a+1)}{\Ga(-a)\Ga(b+1)}
\int_{y/x>1} |y|^{a-c} |y-1|^{c-b-1} |y-x|^{-a}\,dy\nonumber\\
\hphantom{w_4(x;a,b,c)=}{} \
(x\in(-\iy,0)\cup(1,\iy),\;\Re a<1,\;\Re b>0),
\label{31}
\\
w_5(x;a,b,c)=\frac{\Ga(a+b-c+1)}{\Ga(a-c+1)\Ga(b)}
\int_{-\iy}^0 (-y)^{a-c}(1-y)^{c-b-1}(x-y)^{-a}\,dy\nonumber\\
\hphantom{w_5(x;a,b,c)=}{} \
(x\in(0,\iy),\;\Re b>0,\;\Re(a+1)>\Re c),
\label{32}
\\
w_6(x;a,b,c)=\frac{\Ga(c-a-b+1)}{\Ga(1-a)\Ga(c-b)}
\int_{0<\frac{1-y}{1-x}<1} |y|^{a-c} |1-y|^{c-b-1} |x-y|^{-a}\,dy\nonumber\\
\hphantom{w_6(x;a,b,c)=}{} \
(x\in(0,1)\cup(1,\iy),\;\Re(c-b)>0,\;\Re a<1).
\label{18}
\end{gather}

The integrals in \eqref{23}, \eqref{25}, \eqref{28}
(for~$w_1$,~$w_3$,~$w_6$) and in~\eqref{29}, \eqref{31}, \eqref{18}
(for~$w_2$,~$w_4$,~$w_6$)
are of fractional integral type.
Formulas~\eqref{25}, \eqref{29},~\eqref{31} and~\eqref{18}
can be reduced to~\eqref{64} by a~change of integration variable, while~\eqref{23} is a rewritten version
of~\eqref{64}.
Each of the six formulas \eqref{23}--\eqref{28}
is paired with one of the six formulas \eqref{22}--\eqref{18}
by a transformation of integration variable which involves~$x$.
The pairing is as follows:
\begin{alignat*}{4}
& \eqref{23} \ \leftrightarrow \ \eqref{30},\qquad &&  \eqref{24} \ \leftrightarrow \ \eqref{31},\qquad &&
\eqref{25} \ \leftrightarrow \ \eqref{22},& \\
&\eqref{26} \ \leftrightarrow \ \eqref{29}, \qquad &&  \eqref{27} \ \leftrightarrow \  \eqref{32},\qquad &&
\eqref{28} \ \leftrightarrow \  \eqref{18}. &
\end{alignat*}

In 1874 Letnikov \cite[(11), (13), (16), (17), (21), (22)]{17}
(see also Sostak \cite[Section~4]{18})
already gave the Euler type integral
representations \eqref{29}, \eqref{18}, \eqref{31}, \eqref{23}, \eqref{28}, \eqref{25}, respectively,
as solutions of the hypergeometric
dif\/ferential equation \cite[p.~115,~(A)]{17}.
In order to arrive at these results he essentially considered
the fractional integral operator as a transmutation operator
with respect to the hypergeometric dif\/ferential operator.
\\[\smallskipamount]
\centerline{\begin{tabular}{| l l l l l |}
\hline
int.~rep.&transform&from&to&case\\
\hline\hline
\eqref{23}&\eqref{93}&$w_1$&$w_1$&$c+$\\
\eqref{24}&\eqref{37}&$w_3$&$w_2$&$a+$, $c+$\\
\eqref{25}&\eqref{71}&$w_3$&$w_3$&$c+$\\
\eqref{26}&\eqref{89}&$w_1$&$w_4$&$a+$, $c+$\\
\eqref{27}&\eqref{90}&$w_1$&$w_5$&$c+$\\
\eqref{28}&\eqref{82}&$w_6$&$w_6$&$c+$\\
\hline
\eqref{22}&\eqref{91}&$w_4$&$w_1$&$a-$, $c-$\\
\eqref{29}&\eqref{83}&$w_2$&$w_2$&$a-$\\
\eqref{30}&\eqref{87}&$w_6$&$w_3$&$a-$\\
\eqref{31}&\eqref{84}&$w_4$&$w_4$&$a-$, $b-$, $c-$\\
\eqref{32}&\eqref{92}&$w_4$&$w_5$&$a-$\\
\eqref{18}&\eqref{19}&$w_6$&$w_6$&$a-$, $b-$, $c-$\\
\hline
\end{tabular}}

The above 
 table gives for each Euler type integral representation
in the f\/irst column the transformation formula in the second column
from which it can be obtained by specialization of parameters
($b=c$ in the f\/irst six rows and $a=-1$ in the last six rows).
The transformation formula sends $w_i$ in the third column to~$w_j$
in the fourth column. Its case is given in the f\/ifth column.
Formula \eqref{27} is essentially the same as \cite[14.4(9)]{15}.

\section[Generalized Stieltjes transform and fractional integral transform combined]{Generalized Stieltjes transform \\
and fractional integral transform combined}
\label{106}

As observed in \cite[p.~213]{15}, generalized Stieltjes transforms
of dif\/ferent order are connected with each other by fractional integration.
Formula~\eqref{94} below was essentially given there,
and later given with proof in \cite[Theorem~9]{20}.
Some related identities can also be proved:

\begin{Proposition}
Let $\Re(1-\nu)>\Re\mu>0$. Assume that $f\in L_{\rm loc}^1((m,M))$ and
that the integrals on the right-hand side
of the four identities below converge
absolutely. Then
\begin{gather}
\int_{-\iy}^x \left(\int_m^M f(z) \frac{\Ga(1-\nu)}{(z-y)^{1-\nu}}\,dz
\right)\frac{(x-y)^{\mu-1}}{\Ga(\mu)}\,dy
 =\int_m^M f(z) \frac{\Ga(1-\mu-\nu)}{(z-x)^{1-\mu-\nu}}\,dz\quad(x<m),\!\!\!
\label{94}\\
\int_x^\iy \left(\int_m^M f(z) \frac{\Ga(1-\nu)}{(y-z)^{1-\nu}}\,dz\right)
\frac{(y-x)^{\mu-1}}{\Ga(\mu)}\,dy
 =\int_m^M f(z) \frac{\Ga(1-\mu-\nu)}{(x-z)^{1-\mu-\nu}}\,dz\quad(M<x),\!\!\!
\label{95}\\
\int_m^\iy \left(\int_m^y f(z) \frac{(y-z)^{\mu-1}}{\Ga(\mu)}\,dz\right)
\frac{\Ga(1-\nu)}{(x-y)^{1-\nu}}\,dy
 =\int_m^\iy f(z) \frac{\Ga(1-\mu-\nu)}{(z-x)^{1-\mu-\nu}}\,dz\quad(x<m),
\label{96}\\
\int_{-\iy}^M \left(\int_y^M f(z) \frac{(z-y)^{\mu-1}}{\Ga(\mu)}\,dz\right)
\frac{\Ga(1-\nu)}{(y-x)^{1-\nu}}\,dy
 =\int_{-\iy}^M f(z) \frac{\Ga(1-\mu-\nu)}{(x-z)^{1-\mu-\nu}}\,dz\quad(x>M).\!\!\!\!
\label{97}
\end{gather}
\end{Proposition}

The proofs are immediate, by the Fubini theorem and by a version
\cite[(5.12.3)]{3} of the beta integral. Furthermore, \eqref{95}
is an immediate consequence of \eqref{94}, and similarly \eqref{97} of~\eqref{96}.

Examples of formulas \eqref{94}--\eqref{97} can be found by combining
suitable fractional integral formulas in Section~\ref{98} with suitable
generalized Stieltjes transform formulas in Section~\ref{99}.
For instance:
\begin{itemize}\itemsep=0pt
\item
In \eqref{94} let $m=1$, $M=\iy$, $f(z)=z^{c-1} w_6(z;a,b,c)$ and use
\eqref{100} and \eqref{101}.
\item
In \eqref{95} let $m=-\iy$, $M=0$, $f(z)=(-z)^{c-1} w_1(z;a,b,c)$ and use
\eqref{102} and \eqref{90}.
\item
In \eqref{96} let $m=1$, $f(z)=z^{c-1} w_6(z;a,b,c)$ and use
\eqref{82} and \eqref{101}.
\item
In \eqref{97} let $M=0$, $f(z)=(-z)^{c-1} w_1(z;a,b,c)$ and use
\eqref{93} and~\eqref{90}.
\end{itemize}

\subsection*{Acknowledgements}
I am very grateful to Sergei Sitnik for his comments, in particular
about Letnikov's paper~\cite{17} from 1874. Thanks also to
Dmitry Karp for helpful comments. Furthermore, the paper took prof\/it
from comments and lists of typos in referees' reports.

\pdfbookmark[1]{References}{ref}
\LastPageEnding


\begin{thebibliography}{99}
\footnotesize\itemsep=0pt

\bibitem{8}
Andrews G.E., Askey R., Roy R., Special functions, \href{http://dx.doi.org/10.1017/CBO9781107325937}{\textit{Encyclopedia of
  Mathematics and its Applications}}, Vol.~71, Cambridge University Press,
  Cambridge, 1999.

\bibitem{9}
Askey R., Orthogonal polynomials and special functions, Society for Industrial
  and Applied Mathematics, Philadelphia, Pa., 1975.

\bibitem{2}
Askey R., Fitch J., Integral representations for {J}acobi polynomials and some
  applications, \href{http://dx.doi.org/10.1016/0022-247X(69)90165-6}{\textit{J.~Math. Anal. Appl.}} \textbf{26} (1969), 411--437.

\bibitem{1}
Bateman H., The solution of linear dif\/ferential equations by means of def\/inite
  integrals, \textit{Trans. Cambridge Philos. Soc.} \textbf{21} (1909),
  171--196.

\bibitem{5}
Camporesi R., The biradial {P}aley--{W}iener theorem for the {H}elgason
  {F}ourier transform on {D}amek--{R}icci spaces, \href{http://dx.doi.org/10.1016/j.jfa.2014.04.013}{\textit{J.~Funct. Anal.}}
  \textbf{267} (2014), 428--451.

\bibitem{23}
Derezi{\'n}ski J., Hypergeometric type functions and their symmetries,
  \href{http://dx.doi.org/10.1007/s00023-013-0282-4}{\textit{Ann. Henri Poincar\'e}} \textbf{15} (2014), 1569--1653,
  \href{http://arxiv.org/abs/1305.3113}{arXiv:1305.3113}.

\bibitem{4}
Erd\'elyi A., Magnus W., Oberhettinger F., Tricomi F.G., Higher transcendental
  functions, Vol.~I, Mc-Graw Hill, New York, 1953.

\bibitem{15}
Erd\'elyi A., Magnus W., Oberhettinger F., Tricomi F.G., Higher transcendental
  functions, Vol.~II, Mc-Graw Hill, New York, 1953.

\bibitem{26}
Flanders H., Dif\/ferentiation under the integral sign, \href{http://dx.doi.org/10.2307/2319163}{\textit{Amer. Math.
  Monthly}} \textbf{80} (1973), 615--627, {C}orrection,
  \href{http://dx.doi.org/10.2307/2976955}{\textit{Amer. Math. Monthly}}
  \textbf{81} (1974), 145.

\bibitem{20}
Karp D., Prilepkina E., Generalized {S}tieltjes functions and their exact
  order, \textit{J.~Class. Anal.} \textbf{1} (2012), 53--74.

\bibitem{19}
Karp D., Prilepkina E., Hypergeometric functions as generalized {S}tieltjes
  transforms, \href{http://dx.doi.org/10.1016/j.jmaa.2012.03.044}{\textit{J.~Math. Anal. Appl.}} \textbf{393} (2012), 348--359,
  \href{http://arxiv.org/abs/1112.5769}{arXiv:1112.5769}.

\bibitem{6}
Karp D., Sitnik S.M., Inequalities and monotonicity of ratios for generalized
  hypergeometric function, \href{http://dx.doi.org/10.1016/j.jat.2008.10.002}{\textit{J.~Approx. Theory}} \textbf{161} (2009),
  337--352, \href{http://arxiv.org/abs/math.CA/0703084}{math.CA/0703084}.

\bibitem{13}
Kodavanji S., Rathie A.K., Paris R.B., A derivation of two transformation
  formulas contiguous to that of Kummer's second theorem via a dif\/ferential
  equation approach, \href{http://arxiv.org/abs/1501.06173}{arXiv:1501.06173}.

\bibitem{17}
Letnikov A.V., Research related to the theory of integrals of the form
  $\int_0^x (x-u)^{p-1} f(u) du$. Chapter~III. Application to the integration
  of certain dif\/ferential equations, \textit{Mat. Sb.} \textbf{7} (1874),
  111--205 ({i}n {R}ussian).

\bibitem{11}
Lions J.L., Op\'erateurs de {D}elsarte et probl\`emes mixtes, \textit{Bull.
  Soc. Math. France} \textbf{84} (1956), 9--95.

\bibitem{16}
Miller K.S., Ross B., An introduction to the fractional calculus and fractional
  dif\/ferential equations, A Wiley-Interscience Publication, John Wiley \& Sons,
  Inc., New York, 1993.

\bibitem{24}
Miller Jr. W., Lie theory and generalizations of the hypergeometric functions,
  \href{http://dx.doi.org/10.1137/0125026}{\textit{SIAM~J. Appl. Math.}} \textbf{25} (1973), 226--235.

\bibitem{3}
Olver F.W.J., Lozier D.W., Boisvert R.F., Clark C.W. (Editors), N{IST} handbook
  of mathematical functions, U.S. Department of Commerce National Institute of
  Standards and Technology, Washington, DC, Cambridge University Press,
  Cambridge, 2010, {a}vailable at \url{http://dlmf.nist.gov}.

\bibitem{14}
Rainville E.D., Special functions, The Macmillan Co., New York, 1960.

\bibitem{25}
Saito M., Symmetry algebras of normal {${\mathcal A}$}-hypergeometric systems,
  \href{http://dx.doi.org/10.14492/hokmj/1351516751}{\textit{Hokkaido Math.~J.}} \textbf{25} (1996), 591--619.

\bibitem{21}
Sitnik S.M., Transmutations and applications: a~survey,
  \href{http://arxiv.org/abs/1012.3741}{arXiv:1012.3741} ({i}n {R}ussian).

\bibitem{22}
Sitnik S.M., {B}uschman--{E}rdelyi transmutations, classif\/ication and
  applications, \href{http://arxiv.org/abs/1304.2114}{arXiv:1304.2114}.

\bibitem{18}
Sostak R.Ya., Aleksei Vasilevic Letnikov, \textit{Istor.-Mat. Issled.}
  \textbf{5} (1952), 167--238 ({i}n {R}ussian).

\bibitem{12}
Swathi M., Rathie A.K., Paris R.B., A~derivation of two quadratic
  transformations contiguous to that of {G}auss via a~dif\/ferential equation
  approach, \href{http://arxiv.org/abs/1411.5262}{arXiv:1411.5262}.

\bibitem{10}
Szeg{\H{o}} G., Orthogonal polynomials, \textit{American Mathematical Society,
  Colloquium Publications}, Vol.~23, 4th~ed., Amer. Math. Soc., Providence,
  R.I., 1975.

\bibitem{7}
Widder D.V., The {S}tieltjes transform, \href{http://dx.doi.org/10.2307/1989901}{\textit{Trans. Amer. Math. Soc.}}
  \textbf{43} (1938), 7--60.

\end{thebibliography}
\end{document}